\documentclass{amsart}
\usepackage{amscd,amssymb,amsmath,amsthm}
\usepackage[dvips]{graphicx}
\usepackage[all]{xy}
\theoremstyle{plain}
\newtheorem{proposition}{Proposition}
\newtheorem{theorem}[proposition]{Theorem}
\newtheorem{corollary}[proposition]{Corollary}
\newtheorem{lemma}[proposition]{Lemma}
\newtheorem{remark}[proposition]{Remark}
\newtheorem*{proposition*}{Proposition}
\newtheorem*{theorem*}{Theorem}
\newtheorem*{maintheorem*}{Main Theorem}
\newtheorem*{maincorollary*}{Main Corollary}
\newtheorem*{corollary*}{Corollary}
\newtheorem*{lemma*}{Lemma}
\newtheorem*{remark*}{Remark}
\newtheorem*{example*}{Example}

\def\co{\colon\thinspace}

\newcommand{\LN}{\mathcal{L}_0 N}
\newcommand{\LL}{\mathcal{L}_0 L}
\newcommand{\Z}{\mathbb{Z}}

\newcommand{\R}{\mathbb{R}}
\newcommand{\C}{\mathbb{C}}
\newcommand{\CP}{\C\mathbb{P}^{\infty}}
\newcommand{\x}{\mathrm{x}}

\begin{document}

\title[Novikov-symplectic cohomology]{Novikov-symplectic
cohomology and exact Lagrangian embeddings}

\author{Alexander F. Ritter}

\address{Department of Mathematics, M.I.T., Cambridge, MA 02139, USA.}

\email{ritter@math.mit.edu}

\date{version: 24 December 2008}


\begin{abstract}    
Let $N$ be a closed manifold satisfying a mild homotopy
assumption, then for any exact Lagrangian $L\subset T^*N$ the map
$\pi_2(L)\to \pi_2(N)$ has finite index. The homotopy assumption
is either that $N$ is simply connected, or more generally that
$\pi_m(N)$ is finitely generated for each $m\geq 2$. The manifolds
need not be orientable, and we make no assumption on the Maslov
class of $L$.

We construct the Novikov homology theory for symplectic
cohomology, denoted $SH^*(M;\underline{\Lambda}_{\alpha})$, and we
show that Viterbo functoriality holds. We prove that
$SH^*(T^*N;\underline{\Lambda}_{\alpha})$ is isomorphic to the
Novikov homology of the free loopspace. Given the homotopy
assumption on $N$, we show that this Novikov homology vanishes
when $\alpha\in H^1(\LN)$ is the transgression of a non-zero class
in $H^2(\widetilde{N})$. Combining these results yields the above
obstructions to the existence of $L$.
\end{abstract}

\maketitle



\section{Introduction}
Consider a disc cotangent bundle $(DT^*N , d \theta)$ of a closed
manifold $N^n$ together with its canonical symplectic form. We
want to find obstructions to the existence of embeddings $j \co
L^n \hookrightarrow DT^*N$ for which $j^*\theta$ is exact. These
are called \emph{exact Lagrangian embeddings}. For now assume that
all manifolds are orientable and that we use $\Z-$coefficients in
(co)homology.

Denote by $p \co  L \to N$ the composite of $j$ with the
projection to the base. Recall that the ordinary transfer map $p_!
\co H_*(N)\to H_*(L)$ is obtained by Poincar\'e duality and the
pull-back $p^*$, by composing
$$
p_! \co H_*(N) \to H^{n-*}(N) \to H^{n-*}(L) \to H_*(L).
$$

For the space $\LN$ of smooth contractible loops in $N$ such
transfer maps need not exist, as Poincar\'e duality no longer
holds. However, using techniques from symplectic topology, Viterbo
\cite{Viterbo1,Viterbo3} showed that there is a transfer
homomorphism
$$
\mathcal{L}p_! \co H_*(\LN) \to H_*(\LL)
$$
which commutes with the ordinary transfer map for $p$,
$$
\xymatrix{ H_{*}(\LL)
\ar@{<-_{)}}^{c_*}@<1ex>[d]\ar@{<-}[r]^{\mathcal{L}p_!} &
H_*(\LN)\ar@{<-_{)}}^{c_*}@<1ex>[d] \\
H_*(L) \ar@{<-}[r]^{p_!} & H_*(N) }
$$
where $c \co N \to \LN$ denotes the inclusion of constant loops.

For any $\alpha \in H^1(\LN)$, we can define the associated
Novikov homology theory, which is in fact homology with twisted
coefficients in the bundle of Novikov rings $\Lambda=\Z((t))$
associated to a singular cocycle representing $\alpha$. We denote
the bundle by $\underline{\Lambda}_{\alpha}$ and the Novikov
homology by $H_*(\LN;\underline{\Lambda}_{\alpha})$.

\begin{maintheorem*}
For all exact $L \subset T^*N$ and all $\alpha \in H^1(\LN)$,
there exists a commutative diagram
$$
\xymatrix{ H_{*}(\LL;\underline{\Lambda}_{(\mathcal{L}p)^*\alpha})
\ar@{<-}^{c_*}@<1ex>[d] \ar@{<-}[r]^-{\mathcal{L}p_!} &
H_*(\LN;\underline{\Lambda}_{\alpha}) \ar@{<-}^{c_*}@<1ex>[d] \\
H_*(L;c^*\underline{\Lambda}_{(\mathcal{L}p)^*\alpha})
\ar@{<-}[r]^-{p_!} & H_*(N;c^*\underline{\Lambda}_{\alpha}) }
$$
If $c^*\alpha=0$ then the bottom map becomes $p_! \otimes 1 \co
H_*(L) \otimes \Lambda \leftarrow H_*(N) \otimes \Lambda$.
\end{maintheorem*}

Suppose now that $N$ is simply connected. Then a nonzero class
$\beta\in H^2(N)$ defines a nonzero transgression $\tau(\beta)\in
H^1(\LN)$. The associated bundles
$\underline{\Lambda}_{\tau(\beta)}$ on $\LN$ and
$\underline{\Lambda}_{\tau(p^*\beta)}$ on $\mathcal{L}_0 L$
restrict to trivial bundles on $N$ and $L$.

Suppose $\tau(p^*\beta)=0\in H^1(\mathcal{L}_0L)$. Then the above
twisted diagram becomes
$$
\xymatrix{ H_{*}(\LL) \otimes \Lambda \ar@{<-_{)}}_{c_*}@<-1ex>[d]
\ar@{<-}[r]^{\mathcal{L}p_!} &
H_*(\LN;\underline{\Lambda}_{\tau(\beta)}) \ar@{<-}^{c_*}@<1ex>[d] \\
H_*(L) \otimes \Lambda \ar@{<<-}_{q_*}@<-1ex>[u]\ar@{<-}[r]^{p_!}
& H_*(N) \otimes \Lambda }
$$
where $q \co \LN \to N$ is the evaluation at $0$ map. If $N$ is
simply connected and $\beta \neq 0$, then we will show that
$H_*(\LN;\underline{\Lambda}_{\tau(\beta)})=0$, so the fundamental
class $[N] \in H_n(N)$ maps to $c_*[N]=0$. But
$\mathcal{L}p_!(c_*[N])=c_*p_{!}[N]=c_*[L]\neq 0$ since $c_*$ is
injective on $H_*(L)$. Therefore $\tau(p^*\beta)=0$ cannot be
true. This shows that $\tau \circ p^* \co  H^2(N) \to H^1(\LL)$ is
injective. Now, from the commutative diagram
$$
\xymatrix{ H^2(N) \ar@{->}^{p^*}[d]\ar@{->}[r]^-{\tau}_-{\sim} &
\textnormal{Hom}(\pi_2(N),\Z) \cong H^1(\LN) \ar@{->}@<-5ex>^{(\mathcal{L} p)^*}[d] \\
H^2(L) \ar@{->}[r]^-{\tau} & \textnormal{Hom}(\pi_2(L),\Z) \subset
H^1(\mathcal{L}_0 L) }
$$
we deduce that $p^* \co H^2(N)\to H^2(L)$ and
$\textnormal{Hom}(\pi_2(N),\Z) \to \textnormal{Hom}(\pi_2(L),\Z)$
must be injective. Thus we deduce:

\begin{maincorollary*}
If $L \subset T^*N$ is exact and $N$ is simply connected, then the
image of $p_* \co \pi_2(L) \to \pi_2(N)$ has finite index and $p^*
\co H^2(N)\to H^2(L)$ is injective.
\end{maincorollary*}

We emphasize that there is no assumption on the Maslov class of
$L$ in the statement -- this is in contrast to the results of
\cite{Nadler} and \cite{Fukaya-Seidel-Smith}: the vanishing of the
Maslov class is crucial for their argument. Also observe that if
$H^2(N)\neq 0$ then the corollary overlaps with Viterbo's result
\cite{Viterbo3} that there is no exact Lagrangian $K(\pi,1)$
embedded in a simply connected cotangent bundle.

We will prove that the corollary holds even when $N$ and $L$ are
not assumed to be orientable. A concrete application of the
Corollary is that there are no exact tori and no exact Klein
bottles in $T^*S^2$. We will also generalize the Corollary to
obtain a result in the non-simply connected setup:

\begin{corollary*}
Let $N$ be a closed manifold with finitely generated $\pi_m(N)$
for each $m\geq 2$. If $L \subset T^*N$ is exact then the image of
$p_* \co \pi_2(L) \to \pi_2(N)$ has finite index.
\end{corollary*}

This is innovative since in \cite{Fukaya-Seidel-Smith},
\cite{Nadler}, and \cite{Viterbo3} it is crucial that $N$ is
simply connected.

The outline of the proof of the corollary required showing that
the Novikov homology $H_*(\LN;\underline{\Lambda}_{\tau(\beta)})$
vanishes for nonzero $\beta \in H^2(\widetilde{N})$. The idea is
as follows. A class $\tau(\beta)\in
H^1(\mathcal{L}\widetilde{N})=H^1(\LN)$ gives rise to a cyclic
covering $\overline{\LN}$ of $\LN$. Let $t$ be a generator for the
group of deck transformations. The Novikov ring $\Lambda =
\Z((t))=\Z[[t]][t^{-1}]$ is the completion in $t$ of the group
ring $\Z[t,t^{-1}]$ of the cover. The Novikov homology is
isomorphic to $H_*(C_*(\overline{\LN}) \otimes_{\Z[t,t^{-1}]}
\Lambda)$.

Using the homotopy assumptions on $N$ it is possible to prove that
$H_*(\overline{\mathcal{L}_0 N})$ is finitely generated in each
degree. It then easily follows from the flatness of $\Lambda$ over
$\Z[t,t^{-1}]$ and from Nakayama's lemma that
$$
H_*(C_*(\overline{\LN}) \otimes_{\Z[t,t^{-1}]} \Lambda) \cong
H_*(\overline{\LN}) \otimes_{\Z[t,t^{-1}]} \Lambda =0.
$$

The outline of the paper is as follows. In section
\ref{SectionSymplecticCohomology} we recall the construction of
symplectic cohomology and we explain how the construction works
when we use twisted coefficients in the Novikov bundle of some
$\alpha \in H^1(\mathcal{L}N)$, which we call Novikov-symplectic
cohomology. In section \ref{SectionAbbondandoloSchwarzIsomorphism}
we recall Abbondandolo and Schwarz's construction
\cite{Abbondandolo-Schwarz} of the isomorphism between the
symplectic cohomology of $T^*N$ and the singular homology of the
free loopspace $\mathcal{L}N$, and we adapt the isomorphism to
Novikov-symplectic cohomology. In section
\ref{SectionViterboFunctoriality} we review the construction of
Viterbo's commutative diagram, and we show how this carries over
to the case of twisted coefficients. In section
\ref{SectionProofMainTheorem} we prove the main theorem and in
section \ref{SectionProofMainCorollary} we prove the main
corollary. In section
\ref{SubsectionNonSimplyConnectedCotangentBundles} we generalize
the corollary to the case of non-simply connected cotangent
bundles, and in section \ref{SubsectionUnorientedTheory} we extend
the results to the case when $N$ and $L$ are not assumed to be
orientable.
\\

\emph{Acknowledgements:} I would like to thank Paul Seidel for
suggesting this project.
%
%
\section{Symplectic cohomology}\label{SectionSymplecticCohomology}
We review the construction of symplectic cohomology, and refer to
\cite{Viterbo1} for details and to \cite{Seidel} for a survey and
for more references. We assume the reader is familiar with Floer
homology for closed manifolds, for instance see \cite{Salamon}.

\subsection{Liouville domain
setup}\label{SubsectionLiouvilleDomainSetup}
Let $(M^{2n},\theta)$ be a Liouville domain, that is
$(M,\omega=d\theta)$ is a compact symplectic manifold with
boundary and the Liouville vector field $Z$, defined by $i_Z
\omega = \theta$, points strictly outwards along $\partial M$. The
second condition is equivalent to requiring that $\alpha =
\theta|_{\partial M}$ is a contact form on $\partial M$, that is
$d\alpha = \omega|_{\partial M}$ and $\alpha \wedge
(d\alpha)^{n-1} > 0$ with respect to the boundary orientation on
$\partial M$.

The Liouville flow of $Z$ is defined for all negative time $r$,
and it parametrizes a collar $(-\infty,0]\times
\partial M$ of $\partial M$ inside $M$. So we may glue an infinite
symplectic cone $([0,\infty)\times \partial M,d(e^r\alpha))$ onto
$M$ along $\partial M$, so that $Z$ extends to $Z=\partial_r$ on
the cone. This defines the completion $\hat{M}$ of $M$,
$$
\hat{M}=M \cup_{\partial M} [0,\infty)\times \partial M.
$$
We call $(-\infty,\infty)\times \partial M$ the collar of
$\hat{M}$. We extend $\theta$ to the entire collar by $\theta=e^r
\alpha$, and $\omega$ by $\omega=d\theta$. Later on, it will be
convenient to change coordinates from $r$ to $\x=e^r$. The collar
will then be parametrized as the tubular neighbourhood $(0,\infty)
\times \partial M$ of $\partial M$ in $\hat{M}$, where $\partial
M$ corresponds to $\{\x = 1\}$.

Let $J$ be an $\omega-$compatible almost complex structure on
$\hat{M}$ which is of contact type on the collar, that is
$J^*\theta=e^r dr$ or equivalently $J\partial_r = \mathcal{R}$
where $\mathcal{R}$ is the Reeb vector field (we only need this to
hold for $e^r \gg 0$ so that a certain maximum principle applies
there). Denote by $g=\omega(\cdot,J\cdot)$ the $J-$invariant
metric.
%
\subsection{Reeb and Hamiltonian
dynamics}\label{SubsectionReebDynamics}
The Reeb vector field $\mathcal{R} \in C^{\infty}(T \partial M)$
on $\partial M$ is defined by $i_{\mathcal{R}} d\alpha = 0$ and
$\alpha(\mathcal{R})=1$. The periods of the Reeb vector field form
a countable closed subset of $[0,\infty)$.

For $H \in C^{\infty}(\hat{M},\R)$ we define the Hamiltonian
vector field $X_H$ by
$$
\omega(X_H,\cdot) = -dH.
$$
If inside $M$ the Hamiltonian $H$ is a $C^2$-small generic
perturbation of a constant, then the $1$-periodic orbits of $X_H$
inside $M$ are constants corresponding precisely to the critical
points of $H$.

Suppose $H=h(e^r)$ depends only on $e^r$ on the collar. Then $X_H=
h'(e^r) \mathcal{R}$. It follows that every non-constant
$1$-periodic orbit $x(t)$ of $X_H$ which intersects the collar
must lie in $\{ e^r \}\times \partial M$ for some $e^r$ and must
correspond to a Reeb orbit $z(t) = x(t/T) \co [0,T] \to
\partial M$ with period $T = h'(e^r)$. Since the Reeb periods
are countable, if we choose $h$ to have a generic constant slope
$h'(e^r)$ for $e^r \gg 0$ then there will be no $1$-periodic
orbits of $X_H$ outside of a compact set of $\hat{M}$.
%
\subsection{Action functional}
We define the action functional for $x\in C^{\infty}(S^1,M)$ by
$$
A_H(x) = - \int x^*\theta + \int_0^1 H(x(t)) \, dt.
$$
If $H=h(e^r)$ on the collar and $x$ is a $1$-periodic orbit of
$X_H$ in $\{ e^r \} \times
\partial M$, then
$$
A_H(x)=-e^r h'(e^r)+h(e^r).
$$
Let $\mathcal{L}\hat{M} = C^{\infty}(S^1,\hat{M})$ be the space of
free loops in $\hat{M}$. The differential of $A_H$ at $x\in
\mathcal{L}\hat{M}$ in the direction $\xi \in
T_x\mathcal{L}\hat{M} = C^{\infty}(S^1,x^*T\hat{M})$ is
$$
dA_H \cdot \xi = - \int_0^1 \omega(\xi, \dot{x} - X_H) \, dt.
$$
Thus the critical points $x\in \textnormal{Crit}(A_H)$ of $A_H$
are precisely the $1$-periodic Hamiltonian orbits
$\dot{x}(t)=X_H(x(t))$. Moreover, we deduce that with respect to
the $L^2-$metric $\int_0^1 g(\cdot,\cdot) \, dt$ the gradient of
$A_H$ is $\nabla A_H = J(\dot{x} - X_H)$.
%
\subsection{Floer's equation}\label{SubsectionFloersEquation}
For $u \co  \R \times S^1 \to M$, the negative $L^2-$gradient flow
equation $\partial_s u = -\nabla A_H(u)$ in the coordinates $(s,t)
\in \R \times S^1$ is Floer's equation
$$
\partial_s u + J(\partial_t u - X_H) = 0.
$$
The action $A_H(u(s,\cdot))$ decreases in $s$ along Floer
solutions, since
$$
\partial_s(A_H(u(s,\cdot)))=dA_H \cdot \partial_s u = - \int_0^1 \omega(\partial_s u, \partial_t u - X_H) \, dt
= - \int_0^1 |\partial_s u|_g^2 \, dt.
$$
Let $\mathcal{M}'(x_{-},x_{+})$ denote the moduli space of
solutions $u$ to Floer's equation, which at the ends converge
uniformly in $t$ to the $1$-periodic orbits $x_{\pm}$:
$$
\lim_{s \to \pm \infty} u(s,t) = x_{\pm}(t).
$$
These solutions $u$ occur in $\R-$families because we may
reparametrize the $\R$ coordinate by adding a constant. We denote
by $\mathcal{M}(x_{-},x_{+}) = \mathcal{M}'(x_{-},x_{+})/ \R$ the
space of unparametrized solutions.
%
\subsection{Energy}
For a Floer solution $u$ the energy is defined as
$$
E(u) = \int
|\partial_s u|^2 \, ds \, dt = \int \omega(\partial_s u,
\partial_t u - X_H)\, ds \, dt = - \int \partial_s(A_H(u)) \, ds.
$$
Thus for $u \in \mathcal{M}'(x_{-},x_{+})$ there is an a priori
energy estimate,
$$E(u) = A_H(x_{-}) - A_H(x_{+}).
$$
%
%
\subsection{Compactness and the maximum principle}
The only danger in this setup, compared to Floer theory for closed
manifolds, is that there may be Floer trajectories $u \in
\mathcal{M}(x_{-},x_{+})$ which leave any given compact set in
$\hat{M}$. However, for any Floer trajectory $u$, a maximum
principle applies to the function $e^r \circ u$ on the collar,
namely: on any compact subset $\Omega \subset \R \times S^1$ the
maximum of $e^r \circ u$ is attained on the boundary $\partial
\Omega$. Therefore, if the $x_{\pm}$ lie inside $M \cup ([0,R]
\times
\partial M)$ then also all the Floer trajectories in
$\mathcal{M}'(x_{-},x_{+})$ lie in there.

\begin{figure}
\includegraphics[width=0.25\textwidth]{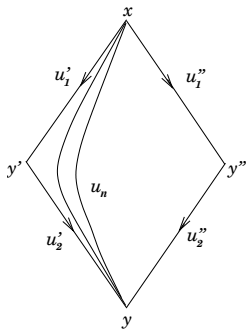}
\caption{The $x,y',y'',y$ are $1$-periodic orbits of $X_H$, the
lines are Floer solutions in $\hat{M}$. The $u_n \in
\mathcal{M}_1(x,y)$ are converging to the broken trajectory
$(u_1',u_2') \in \mathcal{M}_0(x,y')\times
\mathcal{M}_0(y',y)$.}\label{FigureNovikovSymplectic}
\end{figure}
%
\subsection{Transversality and
compactness}\label{SubsectionTransversalityCompactness}
Thanks to the maximum principle and the a priori energy estimates,
the same analysis as for Floer theory for closed manifolds can be
applied to show that for a generic time-dependent perturbation
$(H_t,J_t)$ of $(H,J)$ the corresponding moduli spaces
$\mathcal{M}(x_{-},x_{+})$ are smooth manifolds and have
compactifications $\overline{\mathcal{M}}(x_{-},x_{+})$ whose
boundaries are defined in terms of broken Floer trajectories
(Figure \ref{FigureNovikovSymplectic}). We write
$\mathcal{M}_k(x_{-},x_{+})=\mathcal{M}_{k+1}'(x_{-},x_{+})/\R$
for the $k$-dimensional part of $\mathcal{M}(x_{-},x_{+})$.

The perturbation of $(H,J)$ ensures that the differential
$D\phi_{X_H}^1$ of the time $1$ return map does not have
eigenvalue $1$, where $\phi_{X_H}^t$ is the flow of $X_H$. This
non-degeneracy condition ensures that the $1$-periodic orbits of
$X_H$ are isolated and it is used to prove the transversality
results. In the proofs of compactness, the exactness of $\omega$
is used to exclude the possibility of bubbling-off of
$J-$holomorphic spheres.

To keep the notation under control, we will continue to write
$(H,J)$ even though we are using the perturbed $(H_t,J_t)$
throughout.
%
\subsection{Floer chain complex}\label{SubsectionFloerChainCx}
The Floer chain complex for a Hamiltonian $H \in
C^{\infty}(\hat{M},\R)$ is the abelian group freely generated by
$1$-periodic orbits of $X_H$,
$$
CF^*(H) =\bigoplus \left\{ \Z x : x \in \mathcal{L}\hat{M},\;
\dot{x}(t) = X_H(x(t)) \right\},
$$
and the differential $\partial$ on a generator $y \in
\textnormal{Crit}(A_H)$ is defined as
$$
\partial y = \sum_{u\in \mathcal{M}_0(x,y)} \epsilon(u)\, x,
$$
where $\mathcal{M}_0(x,y)$ is the $0-$dimensional part of
$\mathcal{M}(x,y)$ and the sign $\epsilon(u)\in \{\pm 1\}$ is
determined by the choices of compatible orientations.

We may also filter the Floer complexes by action values $A,B \in
\R \cup \{\pm \infty\}$:
$$
CF^*(H;A,B) =\bigoplus \left\{ \Z x : x \in \mathcal{L}\hat{M},\;
\dot{x}(t) = X_H(x(t)), \; A<A_H(x)<B \right\}.
$$
This is a quotient complex of $CF^*(H)$ if $B\neq \infty$. Observe
that increasing $A$ gives a subcomplex, $CF^*(H;A',B) \subset
CF^*(H;A,B)$ for $A<A'<B$. Moreover there are natural
action-restriction maps $CF^*(H;A,B) \to CF^*(H;A,B')$ for
$A<B'<B$, because the action decreases along Floer trajectories.

Standard methods show that $\partial^2=0$, and we denote by
$HF^*(H)$ and $HF^*(H;A,B)$ the cohomologies of these complexes.
%
\subsection{Continuation maps}\label{SubsectionContinuationMaps}
One might hope that the continuation method of Floer homology can
be used to define a homomorphism between the Floer complexes
$CF^*(H_{-})$ and $CF^*(H_{+})$ obtained for two Hamiltonians
$H_{\pm}$. This involves solving the parametrized version of
Floer's equation
$$
\partial_s u + J_s(\partial_t u - X_{H_s}) = 0,
$$
where $J_s$ are $\omega-$compatible almost complex structures of
contact type and $H_s$ is a homotopy from $H_{-}$ to $H_{+}$ (i.e.
an $s-$dependent Hamiltonian with $(H_s,J_s)=(H_{-},J_{-})$ for $s
\ll 0$ and $(H_s,J_s)=(H_{+},J_{+})$ for $s \gg 0$). If $x$ and
$y$ are respectively $1$-periodic orbits of $X_{H_{-}}$ and
$X_{H_{+}}$, then we can define a moduli space $\mathcal{M}(x,y)$
of such solutions $u$ which converge to $x$ and $y$ at the ends.
This time there is no freedom to reparametrize $u$ in the
$s-$variable.

The action $A_{H_s}(u(s,\cdot))$ along such a solution $u$ will
vary as follows
$$
\partial_s(A_{H_s}(u(s,\cdot)))=- \int_0^1 |\partial_s u|^2 \, dt + \int_0^1 (\partial_sH_s)(u)\,
dt,
$$
so the action decreases if $H_s$ is monotone decreasing,
$\partial_sH_s \leq 0$. The energy is
$$
E(u) = \int |\partial_s u|_{g_s}^2 \, ds \wedge dt =
A_{H_-}(x_{-}) - A_{H_+}(x_{+}) + \int (\partial_s H_s)(u) \, ds
\wedge dt,
$$
so an a priori bound will hold if $\partial_s H_s \leq 0$ outside
of a compact set in $\hat{M}$.

If $H_s=h_s(e^r)$ on the collar and $\partial_s h_s' \leq 0$, then
a maximum principle for $e^r \circ u$ as before will hold on the
collar (we refer to \cite{Seidel} for a very clear proof) and
therefore it automatically guarantees a bound on $(\partial_s
H_s)(u)$ and thus an a priori energy bound.

Thus, if outside of a compact in $\hat{M}$ we have $H_s=h_s(e^r)$
and $\partial_s h_s' \leq 0$, then (after a generic $C^2$-small
time-dependent perturbation of $(H_s,J_s)$) the moduli space
$\mathcal{M}(x,y)$ will be a smooth manifold with a
compactification $\overline{\mathcal{M}}(x,y)$ by broken
trajectories and a continuation map $\phi \co  CF^*(H_{+}) \to
CF^*(H_{-})$ can be defined: on a generator $y \in
\textnormal{Crit}(A_{H_{+}})$,
$$
\phi(y)= \sum_{v\in \mathcal{M}_0(x,y)} \epsilon(v)\, x,
$$
where $\mathcal{M}_0(x,y)$ is the $0-$dimensional part of
$\mathcal{M}(x,y)$ and $\epsilon(v)\in \{ \pm 1\}$ depends on
orientations. Standard methods show that $\phi$ is a chain map and
that these maps compose well: given homotopies from $H_{-}$ to $K$
and from $K$ to $H_{+}$, each satisfying the condition $\partial_s
h_s' \leq 0$ outside of a compact in $\hat{M}$, then the composite
$CF^*(H_{+}) \to CF^*(K) \to CF^*(H_{-})$ is chain homotopic to
$\phi$. So on cohomology, $\phi \co HF^*(H_{+}) \to HF^*(H_{-})$
equals the composite $HF^*(H_{+}) \to HF^*(K) \to HF^*(H_{-})$.

For example, a ``compactly supported homotopy'' is one where $H_s$
is independent of $s$ outside of a compact ($\partial_s H_s =0$
for $s \gg 0$). Continuation maps for $H_s$ and $H_{-s}$ can then
be defined and they will be inverse to each other up to chain
homotopy.
%
\subsection{Symplectic cohomology using only one
Hamiltonian}\label{SubsectionSymplecticCohUsingOneHam}
We change coordinates from $r$ to $\x=e^r$, so the collar is now
$(0,\infty) \times
\partial M \subset \hat{M}$ and $\partial M = \{\x = 1\}$.

Take a Hamiltonian $H_{\infty}$ with $H_{\infty}=h(\x)$ for $\x
\gg 0$, such that $h'(\x) \to \infty$ as $\x \to \infty$. The
symplectic cohomology is defined as the cohomology of the
corresponding Floer complex (after a $C^2$-small time-dependent
perturbation of $(H_{\infty},J)$),
$$
SH^*(M;H_{\infty}) = HF^*(H_{\infty}).
$$
The technical difficulty lies in showing that it is independent of
the choices $(H_{\infty},J)$.
%
\subsection{Symplectic cohomology with action bounds}\label{SubsectionSymplecticCohWithActionBounds}
Similarly one defines the groups
$SH^*(M;H_{\infty};A,B)=HF^*(M;H_{\infty};A,B)$, but these now
depend on the choice of $H_{\infty}$. However, for $B=\infty$,
taking the direct limit as $A \to -\infty$ yields
$$
\lim_{\longrightarrow} SH^*(M;H_{\infty};A,\infty) =
SH^*(M;H_{\infty}),
$$
since $CF^*(H_{\infty};A,\infty)$ are subcomplexes exhausting
$CF^*(H_{\infty};-\infty,\infty)$ as $A \to -\infty$.

If we use action bounds, then it is sometimes possible to vary the
Hamiltonian without using continuation maps. Let $H_1=h_1(\x)$ for
$\x \geq \x_0$, and suppose $A_{h_1}(\x)=-\x h_1'(\x) + h_1(\x) <
A$ for $\x \geq \x_0$. Let $H_2=H_1$ on $M\cup \{ \x \leq \x_0\}$
and $H_2=h_2(\x)$ with $A_{h_2}(\x)<A$ for $\x\geq \x_0$ (e.g. if
$h_2''\geq 0$). Then
$$
CF^*(H_1;A,B) = CF^*(H_2;A,B)
$$
are equal as complexes: the orbits in $\{ \x \geq \x_0\}$ get
discarded by the action bounds; the orbits agree in $M\cup \{\x
\leq \x_0\}$ since $H_1=H_2$ there; and the differential on these
common orbits is the same because the maximum principle forces the
Floer trajectories to lie in $M\cup \{ \x \leq \x_0\}$, where
$H_1=H_2$, so the Floer equations agree.

For example, let $H_1=h_1(\x)=\frac{1}{2}\x^2$ on $\x > 0$, so
$A_{h_1}(\x)=-\frac{1}{2} \x^2$. Take $H_2=H_1$ on $M \cup \{\x
\leq \x_0\}$ and extend $H_2$ linearly on $\{ \x \geq \x_0 \}$.
Then
$CF^*(H_1;-\frac{1}{2}\x_0^2;\infty)=CF^*(H_2;-\frac{1}{2}\x_0^2;\infty)$.
By this trick, $SH^*(M;H_{\infty};A,\infty)$ can be computed by a
Hamiltonian which is linear at infinity, and so
$SH^*(M;H_{\infty})$ can be computed as a direct limit using
Hamiltonians which are linear at infinity and whose slopes at
infinity become steeper and steeper. We now make this precise.
%
\subsection{Hamiltonians linear at
infinity}\label{SubsectionHamiltoniansLinearAtInfty}
Consider Hamiltonians $H$ which equal
$$
h^m_{c,C}(\x) = m(\x-c)+C
$$
for $\x \gg 0$. We assume that the slope $m>0$ does not occur as
the value of the period of any Reeb orbit. If $H_s$ is a homotopy
from $H_{-}$ to $H_{+}$ among such Hamiltonians, i.e.
$H_s=h^{m_s}_{c_s,C_s}(\x)$ for $x \gg 0$, then the maximum
principle (and hence a priori energy bounds for continuation maps)
will hold if
$$
\partial_s \partial_{\x} h^{m_s}_{c_s,C_s} = \partial_s m_s \leq
0.
$$
Suppose that $\partial_s m_s \leq 0$, satisfying $m_s = m_{-}$ for
$s \ll 0$ and $m_s = m_{+}$ for $s\gg 0$, and suppose that the
action values $A_{H_s}(x)$ of $1$-periodic orbits $x$ of $X_{H_s}$
never cross the action bounds $A,B$. Then a continuation map can
be defined,
$$
\phi \co  CF^*(H_{+};A,B) \to CF^*(H_{-};A,B).
$$
These maps compose well: $\phi' \circ \phi''$ is chain homotopic
to $\phi$ (where to define $\phi'$, $\phi''$ we use $m_s'$ varying
from $m_{-}$ to some $m$, $m_s''$ varying from $m$ to $m_{+}$, and
the analogous assumptions as above hold). For example if we vary
only $c,C$, and not $m$, then $\partial_s m_s = 0$ outside of a
compact and the continuation map $\phi$ for $H_s$ can be inverted
(up to chain homotopy) by using the continuation map for $H_{-s}$.
Thus, up to isomorphism, $HF^*(H)$ is independent of the choice of
the constants $c,C$ in $h^m_{c,C}$.
%
\subsection{Symplectic cohomology as a direct
limit}\label{SubsectionSymplCohAsDirectLimit}
Suppose $H_{\infty}=h(\x)$ for $\x \gg 0$ and $h'(\x)\to \infty$
as $\x \to \infty$. Suppose also that $\x h''(\x) > \delta
> 0$ for $x \gg 0$. This implies that $\partial_{\x} A_h =
-\x h''(\x)<-\delta$ so $A_h$ decreases to $-\infty$ as $\x \to
\infty$.

Given $A\in\R$, suppose $A_h(\x)=-\x h'(\x) + h(\x) < A$ for $\x
\geq \x_0$. Define $H = H_{\infty}$ on $M\cup \{\x \leq \x_0 \}$
and extend $H$ linearly in $\x$ for $\x \geq \x_0$. Then
$CF^*(H;A,B)=CF^*(H_{\infty};A,B)$, and $CF^*(H;A,B)$ is a
subcomplex of $CF^*(H_{\infty};-\infty,B)$.

Decreasing $A$ to $A'<A$ defines some Hamiltonian $H'$ which is
steeper at infinity, and it induces a continuation map
$CF^*(H;A,B) \to CF^*(H';A',B)$. The direct limit over these
continuation maps yields a chain isomorphism
$$
\lim_{\longrightarrow} CF^*(H;A,B) \to CF^*(H_{\infty};-\infty,B),
$$
which by the exactness of direct limits induces an isomorphism on
cohomology
$$
\lim_{\longrightarrow} HF^*(H;A,B) \to
SH^*(M;H_{\infty};-\infty,B).
$$
So an alternative definition is
$$
SH^*(M) = \lim_{\longrightarrow} HF^*(H),
$$
where the direct limit is over the continuation maps for all the
Hamiltonians which are linear at infinity, ordered by increasing
slopes $m>0$. In the above argument, we chose particular $H$ which
approximated $H_{\infty}$ on larger and larger compacts. However,
the direct limit can be taken over any family of $H$ with slopes
at infinity $m \to \infty$ because, up to an isomorphism induced
by a continuation map, $HF^*(H)$ is independent of the choice of
$H$ for fixed $m$, so any two cofinal families ($m \to \infty$)
will give the same limit up isomorphism.
%
\subsection{Novikov bundles of coefficients}\label{SubsectionNovikovbundles}
We recommend \cite{Whitehead} as a reference on local systems. Let
$\mathcal{L}N=C^{\infty}(S^1,N)$ denote the free loopspace of a
manifold $N$, and let $\mathcal{L}_0 N$ be the component of
contractible loops. The Novikov ring
$$
\Lambda = \Z((t))=\Z[[t]][t^{-1}]
$$
is the ring of formal Laurent series. Let $\alpha$ be a singular
cocycle representing $a \in H^1(\mathcal{L} N)$. The Novikov
bundle $\underline{\Lambda}_{\alpha}$ is the local system of
coefficients on $\mathcal{L} N$ defined by a copy
$\Lambda_{\gamma}$ of $\Lambda$ over each loop $\gamma \in
\mathcal{L} N$ and by the multiplication isomorphism
$t^{\alpha[u]} \co  \Lambda_{\gamma} \to \Lambda_{\gamma'}$ for
each path $u$ in $\mathcal{L} N$ connecting $\gamma$ to $\gamma'$,
where $\alpha[\cdot] \co  C_1(\mathcal{L} N) \to \Z$ is evaluation
on singular one-chains. A different choice of representative
$\alpha$ for $a$ gives an isomorphic local system, so by abuse of
notation we write $\underline{\Lambda}_{a}$ instead of
$\underline{\Lambda}_{\alpha}$ and $a[u]$ instead of $\alpha[u]$.

We will be using the Novikov bundle
$\underline{\Lambda}_{\tau(\beta)}$ on $\LN$ corresponding to the
transgression $\tau(\beta)\in H^1(\LN)$ of some $\beta\in H^2(N)$
(see \ref{SubsectionTransgressions}). This bundle pulls back to a
trivial bundle under the inclusion of constant loops $c \co  N \to
\LN$, since the transgression $\tau(\beta)$ vanishes on $\pi_1(N)
\subset \pi_1(\LN)$. Therefore we just get ordinary cohomology
with coefficients in the ring $\Lambda$,
$$
H^*(N;c^*\underline{\Lambda}_{\tau(\beta)})\cong H^*(N;\Lambda).
$$
Moreover, for any map $j \co L \to T^*N$ the projection $p \co
L\to T^*N \to N$ induces a map $\mathcal{L} p \co  \mathcal{L}_0 L
\to \LN$, and the pull-back of the Novikov bundle is
$$
(\mathcal{L}p)^*\underline{\Lambda}_{\tau(\beta)} \cong
\underline{\Lambda}_{(\mathcal{L}p)^*(\tau(\beta))} \cong
\underline{\Lambda}_{\tau(p^*\beta)}.
$$
If $\tau(p^*\beta)=0 \in H^1(\mathcal{L}_0 L)$, then this is a
trivial bundle and
$$
H_*(\mathcal{L}_0
L;(\mathcal{L}p)^*\underline{\Lambda}_{\tau(\beta)}) \cong
H_*(\mathcal{L}_0 L)\otimes\Lambda.
$$
%
%
\subsection{Novikov-Floer
cohomology}\label{SubsectionNovikovFloerCohomology}
Let $(M^{2n},\theta)$ be a Liouville domain
(\ref{SubsectionLiouvilleDomainSetup}). Let $\alpha$ be a singular
cocyle representing a class in $H^1(\mathcal{L} M) \cong
H^1(\mathcal{L} \hat{M})$. We define the Novikov-Floer chain
complex for $H \in C^{\infty}(\hat{M},\R)$ with twisted
coefficients in $\underline{\Lambda}_{\alpha}$ to be the
$\Lambda-$module freely generated by the $1$-periodic orbits of
$X_H$,
$$
CF^*(H;\underline{\Lambda}_{\alpha}) =\bigoplus \left\{ \Lambda x
: x \in \mathcal{L} \hat{M},\; \dot{x}(t) = X_H(x(t)) \right\},
$$
and the differential $\delta$ on a generator $y \in
\textnormal{Crit}(A_H)$ is defined as
$$
\delta y = \sum_{u\in \mathcal{M}_0(x,y)} \epsilon(u)\,
t^{\alpha[u]}\,  x,
$$
where $\mathcal{M}_0(x,y)$ and $\epsilon(u)\in \{\pm 1\}$ are the
same as in (\ref{SubsectionFloerChainCx}). The new factor
$t^{\alpha[u]}$ which appears in the differential is precisely the
multiplication isomorphism $\Lambda_{x} \to \Lambda_{y}$ of the
local system $\underline{\Lambda}_{\alpha}$ which identifies the
$\Lambda-$fibres over $x$ and $y$.

As in the untwisted case, we assume that a generic $C^2$-small
time-dependent perturbation of $(H,J)$ has been made so that the
transversality and compactness results of
(\ref{SubsectionTransversalityCompactness}) for the moduli spaces
$\mathcal{M}(x,y)$ are achieved.

\begin{proposition}
$(CF^*(H;\underline{\Lambda}_{\alpha});\delta)$ is a chain
complex, i.e. $\delta\circ\delta = 0$.
\end{proposition}
\begin{proof}
We mimick the proof that $\partial^2=0$ in Floer homology (see
\cite{Salamon}). Observe Figure \ref{FigureNovikovSymplectic}. A
sequence $u_n \in \mathcal{M}_2'(x,y)$ converges to a broken
trajectory $(u_1',u_2')\in \mathcal{M}_1'(x,y') \times
\mathcal{M}_1'(y',y)$, in the sense that there are $s_n\to
-\infty$ and $S_n\to \infty$ with
$$
u_n(s_n+\cdot,\cdot) \to u_1' \textnormal{ and }
u_n(S_n+\cdot,\cdot) \to u_2' \textnormal{ both in }
C^{\infty}_{\textnormal{loc}};
$$
Conversely given such $(u_1',u_2')$ there is a curve $u \co [0,1)
\to \mathcal{M}_2'(x,y)$, unique up to reparametrization and up to
the choice of $u(0)\in\mathcal{M}_2'(x,y)$, which approaches
$(u_1',u_2')$ as $r \to 1$, and the curve is orientation
preserving iff $\epsilon(u_1')\epsilon(u_2')=1$.

So the boundary of $\overline{\mathcal{M}}_1(x,y)$ is parametrized
by $\mathcal{M}_0(x,y') \times \mathcal{M}_0(y',y)$. The value of
$d\alpha=0$ on the connected component of $\mathcal{M}_1(x,y)$
shown in Figure \ref{FigureNovikovSymplectic} is equal to the sum
of the values of $\alpha$ over the broken trajectories,
$$
\alpha[u_1'] + \alpha[u_2'] = \alpha[u_1''] + \alpha[u_2''],
$$
and since
$\epsilon(u_1')\epsilon(u_2')=-\epsilon(u_1'')\epsilon(u_2'')$, we
conclude that
$$
\epsilon(u_1') \, t^{\alpha[u_1']}  \, \epsilon(u_2') \,
t^{\alpha[u_2']} = - \epsilon(u_1'') \, t^{\alpha[u_1'']} \,
\epsilon(u_2'') \,  t^{\alpha[u_2'']}.
$$
Thus the broken trajectories contribute opposite
$\Lambda-$multiples of $x$ to $\delta(\delta y)$ for each
connected component of $\mathcal{M}_1(x,y)$. Hence, summing over
$x,y'$,
$$
\delta(\delta y)=\sum_{(u_1',u_2')\in \mathcal{M}_0(x,y') \times
\mathcal{M}_0(y',y)} \epsilon(u_1') \, t^{\alpha[u_1']}  \,
\epsilon(u_2') \, t^{\alpha[u_2']} \, x=0.\qedhere
$$
\end{proof}
Denote by $HF^*(H;\underline{\Lambda}_{\alpha})$ the
$\Lambda-$modules corresponding to the cohomology groups of the
complex $(CF^*(H;\underline{\Lambda}_{\alpha});\delta)$. We call
these the Novikov-Floer cohomology groups. By filtering the chain
complex by action as in (\ref{SubsectionFloerChainCx}), we can
define
$$
HF^*(H;\underline{\Lambda}_{\alpha};A,B)=H^*(CF^*(H;\underline{\Lambda}_{\alpha};A,B);\delta).
$$
%
%
\subsection{Twisted continuation
maps}\label{SubsectionTwistedContinuationMaps}
%
We now show that the continuation method described in
(\ref{SubsectionContinuationMaps}) can be used in the twisted case
under the same assumptions that we made in the untwisted case.
Recall that this involves solving
$$
\partial_s v + J_s(\partial_t v - X_{H_s}) = 0,
$$
and that under suitable assumptions on $(H_s,J_s)$ the moduli
spaces $\mathcal{M}(x,y)$ of solutions $v$ joining $1$-periodic
orbits $x,y$ of $X_{H_{-}}$ and $X_{H_{+}}$ are smooth manifolds
with compactifications $\overline{\mathcal{M}}(x,y)$ whose
boundaries are given by broken trajectories.

So far, using a twisted differential does not change the setup.
However, to make the continuation map $\phi \co
CF^*(H_{+};\underline{\Lambda}_{\alpha}) \to
CF^*(H_{-};\underline{\Lambda}_{\alpha})$ into a chain map we need
to define it on a generator $y \in \textnormal{Crit}(A_{H_{+}})$
by
$$
\phi(y)= \sum_{v \in \mathcal{M}_0(x,y)} \epsilon(v)\,
t^{\alpha[v]} \, x,
$$
where $\mathcal{M}_0(x,y)$ and $\epsilon(v)\in \{ \pm 1\}$ are as
in (\ref{SubsectionContinuationMaps}).

\begin{proposition}
$\phi \co  CF^*(H_{+};\underline{\Lambda}_{\alpha}) \to
CF^*(H_{-};\underline{\Lambda}_{\alpha})$ is a chain map.
\end{proposition}
\begin{proof}
\begin{figure}
\includegraphics[width=0.3\textwidth]{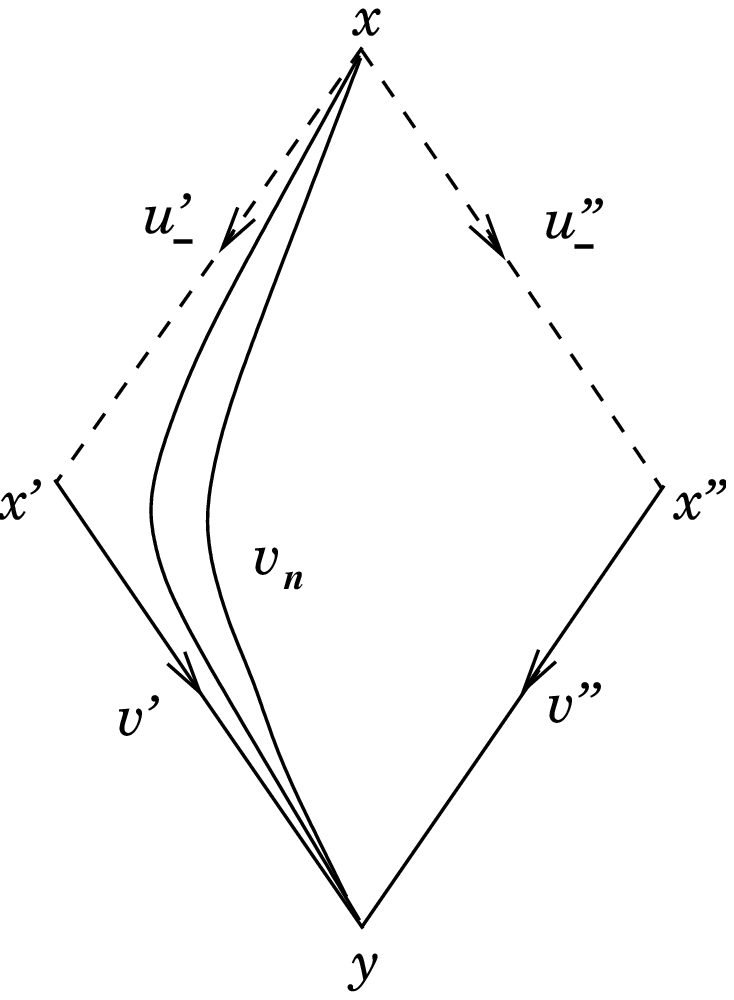}\quad
\includegraphics[width=0.3\textwidth]{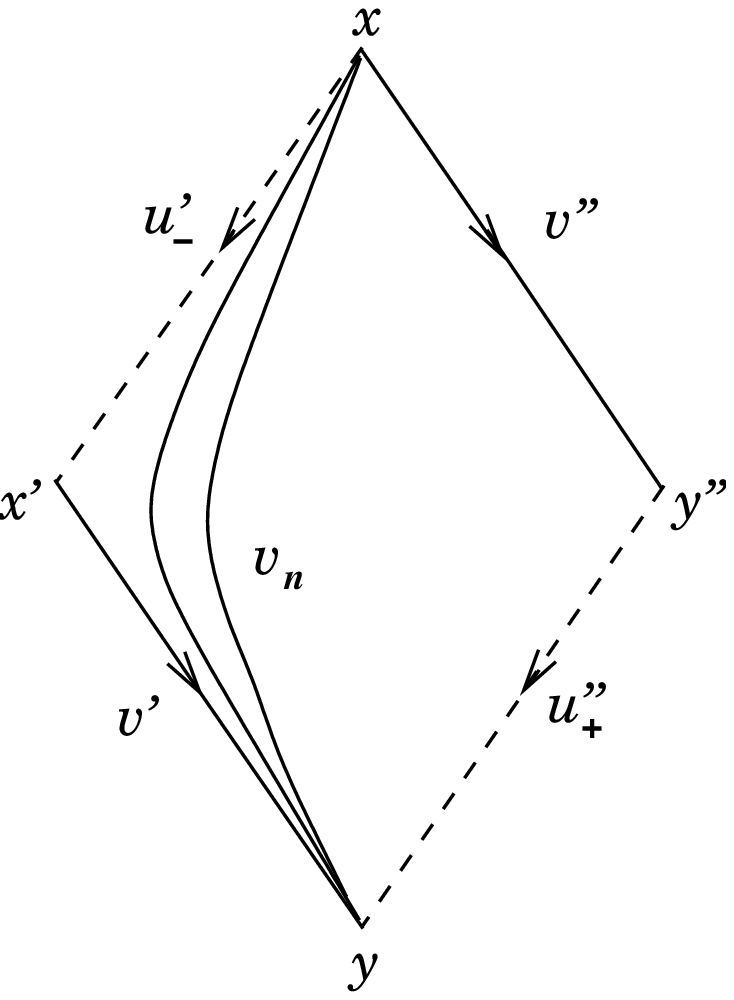}\quad
\includegraphics[width=0.3\textwidth]{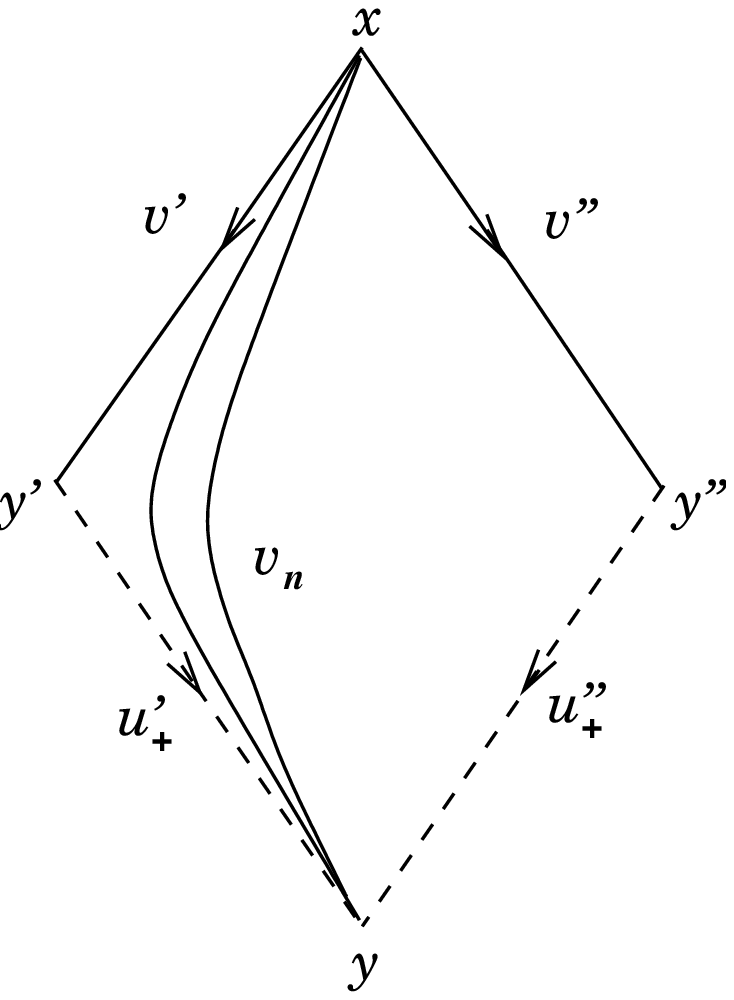}
\caption{The dashed lines $u_{\pm}$ are Floer solutions converging
to $1$-periodic orbits of $X_{H_{\pm}}$, the solid lines are
continuation map solutions, the $v_n \in \mathcal{M}_1(x,y)$ are
converging to broken
trajectories.}\label{FigureNovikovSymplectic2}
\end{figure}

We mimick the proof that $\phi$ is a chain map in the untwisted
case \cite{Salamon}. Denote by $\mathcal{M}^{H_\pm}(\cdot,\cdot)$
the moduli spaces of Floer trajectories for $H_{\pm}$. Observe
Figure \ref{FigureNovikovSymplectic2}.

A compactness result in Floer homology shows that a sequence of
solutions $v_n \in \mathcal{M}_1(x,y)$ will converge to a broken
trajectory
$$
(u'_{-},v')\in \mathcal{M}_0^{H_-}(x,x') \times
\mathcal{M}_0(x',y) \;\textnormal{ or }\; (v',u'_{+})\in
\mathcal{M}_0(x,y') \times \mathcal{M}_0^{H_+}(y',y).
$$
Conversely, given such $(u'_{-},v')$ or $(v',u'_{+})$ there is a
smooth curve $v \co [0,1) \to \mathcal{M}_1(x,y)$, unique up to
reparametrization and up to the choice of $v(0)$, which approaches
the given broken trajectory as $r \to 1$, and the curve is
orientation preserving iff respectively
$\epsilon(u'_{-})\epsilon(v')=-1$ and
$\epsilon(v')\epsilon(u'_{+})=1$.

Thus the boundary of $\overline{\mathcal{M}}_1(x,y)$ is
parametrized by $-\mathcal{M}_0^{H_-}(x,x') \times
\mathcal{M}_0(x',y)$ and by $\mathcal{M}_0(x,y') \times
\mathcal{M}_0^{H_+}(y',y)$. The value of $d\alpha=0$ on a
connected component of $\mathcal{M}_1(x,y)$ as in Figure
\ref{FigureNovikovSymplectic2} is equal to the sum of the values
of $\alpha$ over the broken trajectories. For instance, in the
second figure
$$
\alpha[u_{-}'] + \alpha[v'] = \alpha[v''] + \alpha[u_{+}''],
$$
and since
$\epsilon(u_{-}')\epsilon(v')=\epsilon(v'')\epsilon(u_{+}'')$,
$$
\epsilon(u_{-}') \, t^{\alpha[u_{-}']}  \, \epsilon(v') \,
t^{\alpha[v']} = \epsilon(v'') \, t^{\alpha[v'']} \,
\epsilon(u_{+}'') \,  t^{\alpha[u_{+}'']}.
$$
Thus the broken trajectories contribute equal $\Lambda-$multiples
of $x$ to $\delta(\phi(y))$ and $\phi(\delta y)$ for that
component of $\mathcal{M}_1(x,y)$. A similar computation shows
that in the first or third figures, the two broken trajectories
contribute opposite $\Lambda-$multiples of $x$ and so in total
give no contribution to $\delta(\phi(y))$ or $\phi(\delta y)$. We
deduce that
$$
\begin{array}{l}
\delta(\phi(y)) =  \displaystyle \sum_{(u_{-}',v')\in
\mathcal{M}_0^{H_-}(x,x') \times \mathcal{M}_0(x',y)}
\epsilon(u_{-}') \, t^{\alpha[u_{-}']} \, \epsilon(v') \,
t^{\alpha[v']} \, x \; = \\
\qquad\;\,\quad = \displaystyle \sum_{(v',u_{+}')\in
\mathcal{M}_0(x,y') \times \mathcal{M}_0^{H_+}(y',y)} \epsilon(v')
\, t^{\alpha[v']} \, \epsilon(u_{+}') \,  t^{\alpha[u_{+}']} \, x
\; = \phi(\delta y),
\end{array}
$$
where we sum respectively over $x,x'$ and $x,y'$. Hence $\phi$ is
a chain map.
\end{proof}

A similar argument, by mimicking the proof of the untwisted case,
shows that the twisted continuation maps compose well: given
homotopies from $H_{-}$ to $K$ and from $K$ to $H_{+}$ satisfying
the conditions required in the untwisted case, the composite
$CF^*(H_{+};\underline{\Lambda}_{\alpha}) \to
CF^*(K;\underline{\Lambda}_{\alpha}) \to
CF^*(H_{-};\underline{\Lambda}_{\alpha})$ is chain homotopic to
$\phi$.
%
\subsection{Novikov-symplectic
cohomology}\label{SubsectionNovikovSymplecticCohomology}
If we use the groups $HF^*(H;\underline{\Lambda}_{\alpha})$ from
(\ref{SubsectionNovikovFloerCohomology}) in place of $HF^*(H)$ in
our discussion
(\ref{SubsectionSymplecticCohUsingOneHam}-\ref{SubsectionSymplCohAsDirectLimit})
of the symplectic cohomology groups of a Liouville domain, and we
use the twisted continuation maps constructed in
(\ref{SubsectionTwistedContinuationMaps}), then we obtain the
$\Lambda-$modules
$$
SH^*(M;H_{\infty};\underline{\Lambda}_{\alpha}) \, \textnormal{
and } \, SH^*(M;H_{\infty};\underline{\Lambda}_{\alpha};A,B),
$$
which we call Novikov-symplectic cohomology groups.

So for $H_{\infty}$ such that $H_{\infty}=h(\x)$ for $\x \gg 0$
and $h'(\x) \to \infty$ as $\x \to \infty$, we define
$$
SH^*(M;\underline{\Lambda}_{\alpha}) =
HF^*(H_{\infty};\underline{\Lambda}_{\alpha}).
$$
Alternatively, we may use the Hamiltonians $H$ which equal
$h^m_{c,C}(\x)= m(\x-c) + C$ for $\x \gg 0$, and we take the
direct limit over the twisted continuation maps between the
corresponding twisted Floer cohomologies as the slopes $m>0$
increase,
$$
SH^*(M;\underline{\Lambda}_{\alpha}) = \lim_{\longrightarrow}
HF^*(H;\underline{\Lambda}_{\alpha}).
$$
%
%
%

\section{Abbondandolo-Schwarz isomorphism}\label{SectionAbbondandoloSchwarzIsomorphism}

For a closed (oriented) manifold $N^n$, the symplectic cohomology
of the cotangent disc bundle $M^{2n}=DT^*N$ is isomorphic to the
homology of the free loopspace,
$$
SH^*(DT^*N) \cong H_{n-*}(\mathcal{L} N).
$$
This was first proved by Viterbo \cite{Viterbo2}, and there are
now two alternative approaches by Abbondandolo-Schwarz
\cite{Abbondandolo-Schwarz} and Salamon-Weber
\cite{Salamon-Weber}. We will use the Abbondandolo-Schwarz
isomorphism and show that it carries over to twisted coefficients,
but similar arguments could be carried out using either of the
other approaches. We will recall the construction
\cite{Abbondandolo-Schwarz} of the chain isomorphism
$$
(CM_*(\mathcal{E}),\partial^\mathcal{E}) \to
(CF^{n-*}(H),\partial^H),
$$
between the Morse complex of the Hilbert manifold $\mathcal{L}^1N
= W^{1,2}(S^1,N)$ with respect to a certain Lagrangian action
functional $\mathcal{E}$ and the Floer complex of $T^*N$ with
respect to an appropriate Hamiltonian $H\in C^{\infty}(S^1\times
T^*N,\R)$.

Let $\pi \co  T^*N \to N$ denote the projection. We use the
standard symplectic structure $\omega = d\theta$ and Liouville
field $Z$ on $T^*N$, which in local coordinates $(q,p)$ are
$$
\theta = p\,dq \qquad \omega = dp \wedge dq \qquad Z=p
\,\partial_p.
$$
A metric on $N$ induces metrics and Levi-Civita connections on
$TN$ and $T^*N$, and it defines a splitting $T_{(q,p)}T^*N \cong
T_q N \oplus T_q^*N \cong T_q N \oplus T_q N$ into horizontal and
vertical vectors and a connection $\nabla = \nabla_q \oplus
\nabla_p$, and similarly for $T_{(q,v)}TN$. For this splitting our
preferred $\omega-$compatible almost complex structure is $J
\partial_q =-\partial_p$.

\begin{remark*}
Our action $A_H$ is opposite to the action $\mathcal{A}$ used in
\cite{Abbondandolo-Schwarz}, so our Floer trajectory $u(s,t)$
corresponds to $u(-s,t)$ in \cite{Abbondandolo-Schwarz}. Our
grading is $\mu(x)=n-\mu_{CZ}(x)$ (see \cite{Salamon}, where the
sign of $H$ is opposite to ours), the one used in
\cite{Abbondandolo-Schwarz} is $\mu_{CZ}(x)$ and that in
\cite{Seidel} is $-\mu_{CZ}(x)$. In our convention the index
$\mu(x)$ agrees with the Morse index $\textnormal{ind}_{H}(x)$ for
$x\in \textnormal{Crit}(H)$ when $H$ is a $C^2$-small Morse
Hamiltonian.
\end{remark*}
%
\subsection{The Lagrangian Morse
functional}\label{SubsectionLagrangianMorseFunctional}
The Morse function one considers on $\mathcal{L}^1 N=
W^{1,2}(S^1,N)$ is the Lagrangian action functional
$$
\mathcal{E}(q) = \int_0^1 L(t,q(t),\dot{q}(t)) \, dt,
$$
where the Lagrangian $L \in C^{\infty}(S^1 \times TN, \R)$ is
generic and satisfies certain growth conditions and a strong
convexity assumption that ensure that: $\mathcal{E}$ is bounded
below; the critical points of $\mathcal{E}$ are non-degenerate
with finite Morse index; and $\mathcal{E}$ satisfies the
Palais-Smale condition (any sequence of $q_n \in \mathcal{L}^1N$
with bounded actions $\mathcal{E}(q_n)$ and with energies $\|
\nabla \mathcal{E}(q_n) \|_{W^{1,2}}\to 0$ has a convergent
subsequence). By an appropriate generic perturbation it is
possible to obtain a metric $G$ which is uniformly equivalent to
the $W^{1,2}$ metric on $\mathcal{L}^1 N$ and for which
$(\mathcal{E},G)$ is a Morse-Smale pair. Denote by
$\mathcal{M}^{\mathcal{E}}(q_-,q_+)=\mathcal{M}_{\mathcal{E}}'(q_-,q_+)/\R$
the unparametrized trajectories, where
$$
\mathcal{M}_{\mathcal{E}}'(q_-,q_+) = \{ v \co \R \to
\mathcal{L}^1N:
\partial_s v(s)=-\nabla \mathcal{E}(v(s)),\; \lim_{s\to \pm
\infty} v(s) = q_{\pm} \}.
$$
Under these assumptions, infinite dimensional Morse theory can be
applied to $(\mathcal{L}^1 N,\mathcal{E},G)$ and the Morse
homology is isomorphic to the singular homology of $\mathcal{L}^1
N$ (which is isomorphic to the singular homology of $\mathcal{L}
N$, since $\mathcal{L}^1 N$ and $\mathcal{L} N$ are homotopy
equivalent). This isomorphism respects the filtration by action:
the homology of the Morse complex generated by the $x\in
\textnormal{Crit}(\mathcal{E})$ with $\mathcal{E}(x)<a$ is
isomorphic to $H_*(\{ q\in \mathcal{L}^1N: \mathcal{E}(q)<a \})$.
The isomorphism also respects the splitting of the Morse complex
and the singular complex into subcomplexes corresponding to the
components of $\mathcal{L}^1N$ (which are indexed by the conjugacy
classes of $\pi_1(N)$).
%
\subsection{Legendre transform}\label{SubsectionLegendreTransform}
$L$ defines a Hamiltonian $H \in C^{\infty}(S^1 \times T^*N,\R)$
by
$$
H(t,q,p) = \max_{v\in T_q N} (p\cdot v - L(t,q,v)).
$$
The strong convexity assumption on $L$ ensures that there is a
unique maximum precisely where $p=d_vL(t,q,v)$ is the differential
of $L$ restricted to the vertical subspace
$T^{\textnormal{vert}}_{(q,v)}TN \cong T_q N$, and it ensures that
the Legendre transform
$$
\mathfrak{L} \co S^1\times TN \to S^1\times T^*N,\, (t,q,v)
\mapsto (t,q,d_v L(t,q,v))
$$
is a fiber-preserving diffeomorphism.

Pull back $(\omega,H,X_H)$ via $\mathfrak{L}$ to obtain
$(\mathfrak{L^*}\omega, H\circ \mathfrak{L}, Y_L)$, so
$\mathfrak{L^*}\omega(Y_L,\cdot)=-d(H\circ \mathfrak{L})$. The
critical points of $\mathcal{E}$ are precisely the $1$-periodic
orbits $(q,\dot{q})$ of $Y_L$ in $TN$, and these bijectively
correspond to $1$-periodic orbits $x$ of $X_H$ in $T^*N$ via
$$
(t,x)=\mathfrak{L}(t,q,\dot{q}).
$$
Under this correspondence the Morse index of $q$ is
$m(q)=n-\mu(x)$ (in the conventions of
\cite{Abbondandolo-Schwarz}, $m(q)=\mu_{CZ}(x)$). Moreover, for
any $W^{1,2}-$path $x \co [0,1]\to T^*N$,
$$
\mathcal{E}(\pi x) \geq -A_H(x),
$$
which becomes an equality iff $(t,x)=\mathfrak{L}(t,\pi
x,\partial_t (\pi x))$ for all $t$.
%
\subsection{The moduli spaces $\mathbf{\mathcal{M}^+(q,x)}$}
For $1$-periodic orbits $q$ of $Y_L$ and $x$ of $X_H$, define
$\mathcal{M}^{+}(q,x)$ to be the collection of all maps $u \in
C^{\infty}((-\infty,0) \times S^1, T^*N)$ which are of class
$W^{1,3}$ on $(-1,0)\times S^1$ and which solve Floer's equation
$$
\partial_s u + J(t,u)(\partial_t u - X_H(t,u)) = 0,
$$
with the following boundary conditions:

 \indent i) as $s\to
-\infty$, $u(s,\cdot)\to x$ uniformly in $t$;

\indent ii) as $s\to 0$, $u$ will converge to some loop
$u(0,\cdot)$ of class $W^{2/3,3}$, and we require that the
projection $\overline{q}(t)=\pi\circ u(0,t)$ in $N$ flows backward
to $q\in \textnormal{Crit}(\mathcal{E})$ along the negative
gradient flow $\phi_{-\nabla \mathcal{E}}^s$ of $\mathcal{E}$: $
\phi_{-\nabla \mathcal{E}}^s (\overline{q}) \to q \textnormal{ as
} s\to -\infty$.

Loosely speaking, $\mathcal{M}^+(q,x)$ consists of pairs of
trajectories $(w,u_+)$ where $w$ is a $-\nabla \mathcal{E}$
trajectory in $N$ flowing out of $q$, and $u_+$ is a Floer
solution in $T^*N$ flowing out of $x$, such that $w$ and $\pi u_+$
intersect in a loop $\overline{q}(t)=\pi u_+(0,t)$ in $N$.
%
\subsection{Transversality and compactness}
The assumption on $H$ and $L$ is that there are constants $c_i>0$
such that for all $(t,q,p) \in S^1 \times T^*N$, $(t,q,v) \in S^1
\times TN$,
$$
\begin{array}{lll}
dH(p\partial_p) - H \geq c_0 |p|^2 - c_1, & |\nabla_p H| \leq
c_2(1+|p|), & |\nabla_q H| \leq
c_2(1+|p|^2); \\
\nabla_{vv} L \geq c_3\, \textnormal{Id}, \;\, |\nabla_{vv} L|
\leq c_4, & |\nabla_{qv} L| \leq c_4 (1+|v|), & |\nabla_{qq} L|
\leq c_4 (1+|v|^2).
\end{array}
$$
We also assume that a small generic perturbation of $L$ (and hence
$H$) are made so that the nondegeneracy condition (see
\ref{SubsectionTransversalityCompactness}) holds for $1$-periodic
orbits of $Y_L$ and $X_H$. We call such $H,L$ regular. For regular
$H$, there are only finitely many $1$-periodic orbits $x$ of $X_H$
with action $A_H(x) \geq a$, for $a\in \R$. After a small generic
perturbation of $J$, the compactness and transversality results of
(\ref{SubsectionTransversalityCompactness}) hold for the spaces
$\mathcal{M}^{H}(x,y)=\mathcal{M}'(x,y)/\R$ of unparametrized
Floer solutions in $T^*N$ converging to $x,y\in
\textnormal{Crit}(A_H)$ at the ends, and similar results hold for
$\mathcal{M}^{+}(q,x)$ by using the $W^{1,3}$ condition in the
definition to generalize the proofs used for $\mathcal{M}'(x,y)$.

When all of the above assumptions are satisfied, we call
$(L,G,H,J)$ regular. In this case,
$\mathcal{M}^{\mathcal{E}}(p,q)$, $\mathcal{M}^{H}(x,y)$ and
$\mathcal{M}^{+}(q,x)$ are smooth manifolds with compactifications
by broken trajectories, and their dimensions are:
$$
\begin{array}{l} \textnormal{dim }\mathcal{M}^{\mathcal{E}}(p,q) =
m(p)-m(q)-1
\\ \textnormal{dim } \mathcal{M}^{H}(x,y) = \mu(x) -
\mu(y)-1
\\ \textnormal{dim }\mathcal{M}^{+}(q,x) = m(q)+\mu(x)-n
\end{array}
$$
and we denote by $\mathcal{M}_k^{\mathcal{E}}(p,q)$,
$\mathcal{M}_k^{H}(x,y)$ and $\mathcal{M}_k^{+}(q,x)$ the
$k$-dimensional ones.

\begin{theorem*}[Abbondandolo-Schwarz
{\cite{Abbondandolo-Schwarz}}] If $(L,G,H,J)$ is regular then
there is a chain-complex isomorphism $ \varphi \co
(CM_*(\mathcal{E}),\partial^\mathcal{E}) \to
(CF^{n-*}(H),\partial^H)$, which on a generator $q \in
\textnormal{Crit}(\mathcal{E})$ is defined as
$$
\varphi(q) = \sum_{u_+ \in \mathcal{M}_0^{+}(q,x)} \epsilon(u_+)
\, x,
$$
where $\epsilon(u_+)\in \{\pm 1\}$ are orientation signs. The
isomorphism is compatible with the splitting into subcomplexes
corresponding to different conjugacy classes of $\pi_1(N)$, and it
is compatible with the action filtrations: for any $a\in\R$ it
induces an isomorphism on the subcomplexes generated by the $q,x$
with $\mathcal{E}(q)<a$ and $-A_H(x)<a$.
\end{theorem*}
%
\subsection{Proof that $\varphi$ is an isomorphism.}
Since actions decrease along orbits and $\mathcal{E}(\pi x) \geq
-A_H(x)$ with equality iff $(t,x)=\mathfrak{L}(t,\pi \circ
x,\partial_t \pi x)$, we deduce that
$$
\mathcal{E}(q) \geq \mathcal{E}(\overline{q}) \geq
-A_H(u_+(0,\cdot)) \geq -A_H(x),
$$
so $\mathcal{E}(q) \geq -A_H(x)$ with equality iff $q \equiv
\overline{q}$, $u_+ \equiv x$, $q = \pi x$ and
$(t,x)=\mathfrak{L}(t,q,\dot{q})$. Therefore if
$\mathcal{E}(q)<-A_H(x)$ then $\mathcal{M}^{+}(q,x)=\emptyset$,
and if $\mathcal{E}(q)=-A_H(x)$ then $\mathcal{M}^{+}(q,x)$ is
either empty or, when $(t,x)=\mathfrak{L}(t,q,\dot{q})$, it
consists of $u_+ \equiv x$. Now order the generators of
$CM_*(\mathcal{E})$ according to increasing action and those of
$CF^*(H)$ according to decreasing action, and so that the order is
compatible with the correspondence
$(t,x)=\mathfrak{L}(t,q,\dot{q})$. Then $\varphi$ is a (possibly
infinite) upper triangular matrix with $\pm 1$ along the diagonal,
so $\varphi$ is an isomorphism.
%
\subsection{Proof that $\varphi$ is a chain map}
The differentials for the complexes
$(CM_*(\mathcal{E}),\partial^\mathcal{E})$ and
$(CF^*(H),\partial^H)$  are defined on generators $q\in
\textnormal{Crit}(\mathcal{E})$, $y \in \textnormal{Crit}(A_H)$ by
$$
\partial^{\mathcal{E}} (q) = \sum_{v \in \mathcal{M}_0^{\mathcal{E}}(q,p)} \epsilon(v)\,
p \quad \quad \quad \quad
\partial^{H} (y) = \sum_{u \in \mathcal{M}_0^{H}(x,y)} \epsilon(u)\, x
$$
where $\epsilon(v),\epsilon(u)\in \{\pm 1\}$ depend on
orientations.
\begin{figure}
\includegraphics[width=0.3\textwidth]{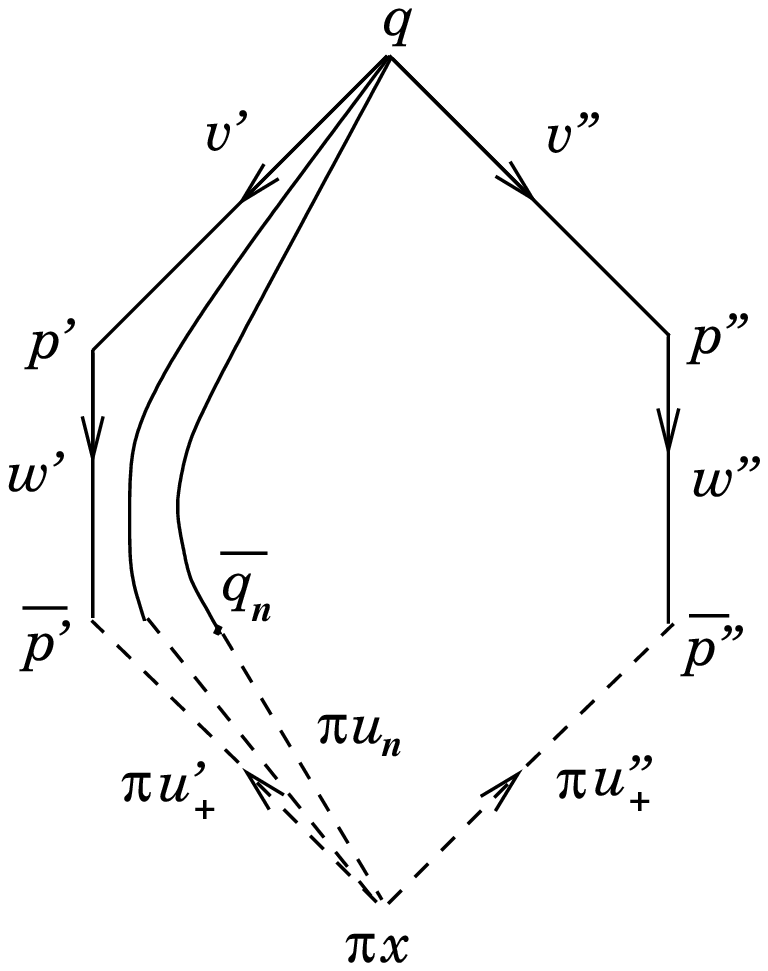}\quad
\includegraphics[width=0.3\textwidth]{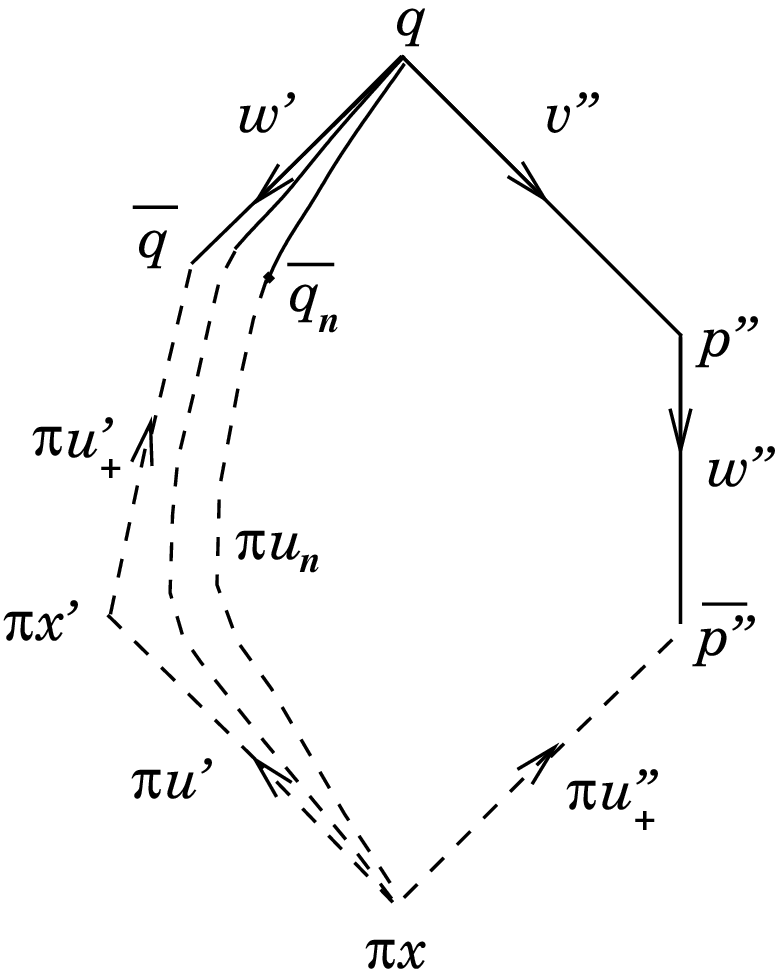}\quad
\includegraphics[width=0.3\textwidth]{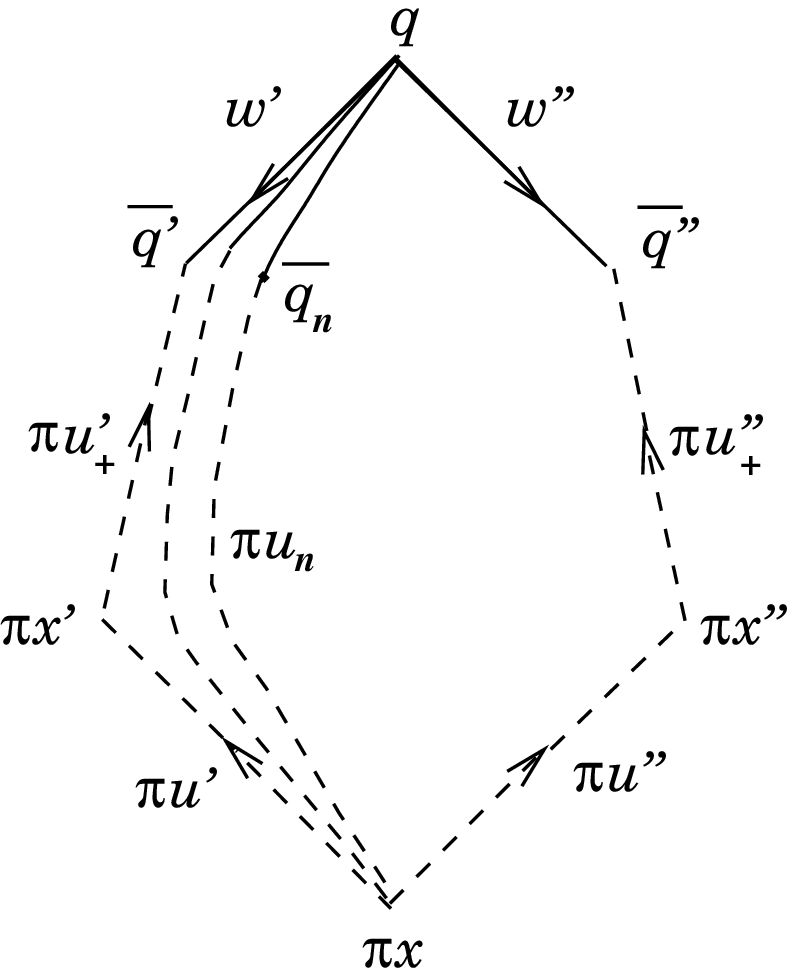}
\caption{Solid lines are $-\nabla \mathcal{E}$ trajectories in
$N$, dotted lines are the projections under $\pi \co T^*N \to N$
of Floer solutions.}\label{FigureAbbondandoloSchwarz}
\end{figure}
Observe Figure \ref{FigureAbbondandoloSchwarz}. A compactness
argument shows that the broken trajectories that compactify
$\mathcal{M}_1^{+}(q,x)$ are of two types: either (\emph{i}) the
$-\nabla\mathcal{E}$ trajectory breaks, or (\emph{ii}) the Floer
trajectory breaks. More precisely, if $u_n \in
\mathcal{M}_1^{+}(q,x)$ and $\overline{q}_n(t)=\pi (u_n(0,t))$,
then either
\\
\indent (\emph{i}) there are $[v] \in
\mathcal{M}_0^{\mathcal{E}}(q,p)$; $u'_+ \in
\mathcal{M}_0^{+}(p,x)$; and reals $t_n\to -\infty$ with
$$
\phi_{-\nabla\mathcal{E}}^{t_n}(\overline{q}_n) \to v(0)
\textnormal{ in } W^{1,2}, \textnormal{ and } u_n \to u'_+
\textnormal{ in } C^{\infty}_{\textnormal{loc}};
$$
\indent (\emph{ii}) or there are $[u'] \in \mathcal{M}_0^H(x,x')$;
$u_+' \in \mathcal{M}_0^{+}(q,x')$; and reals $s_n\to -\infty$
with
$$
u_n(s_n+\cdot,\cdot) \to u' \textnormal{ and } u_n \to u_+'
\textnormal{ both in } C^{\infty}_{\textnormal{loc}}.
$$
Conversely, given $(v,u'_+)$ or $(u',u_+')$ as above, there is a
smooth curve $u \co [0,1) \to \mathcal{M}_1^+(q,x)$, unique up to
reparametrization and up to the choice of $u(0)$, which approaches
the given broken trajectory as $r \to 1$, and the curve is
orientation preserving iff respectively
$\epsilon(v)\epsilon(u'_+)=1$ and $\epsilon(u')\epsilon(u_+')=-1$.

Thus the boundary of $\mathcal{M}_1^{+}(q,x)$ is parametrized by
$\mathcal{M}_0^{\mathcal{E}}(q,p) \times \mathcal{M}_0^{+}(p,x)$
and by $-\mathcal{M}_0^H(x,x') \times \mathcal{M}_0^{+}(q,x')$.
Figure \ref{FigureAbbondandoloSchwarz} shows the possible
components of $\mathcal{M}_1^{+}(q,x)$: in the first and third
figures, the broken trajectories contribute zero respectively to
$\varphi(\partial^{\mathcal{E}}(q))$ and
$\partial^{H}(\varphi(q))$; in the second figure we see that
$\epsilon(u')\epsilon(u_+')=\epsilon(v'')\epsilon(u''_+)$, so the
broken trajectories contribute $\pm x$ to both
$\partial^{H}(\varphi(q))$ and
$\varphi(\partial^{\mathcal{E}}(q))$. Therefore
$\partial^{H}(\varphi(q))=\varphi(\partial^{\mathcal{E}}(q))$, so
$\varphi$ is a chain map.
%
\subsection{The twisted version of the Abbondandolo-Schwarz
isomorphism}
Let $\alpha$ be a singular cocyle representing a class in
$H^1(\mathcal{L}^1 N) \cong H^1(\mathcal{L} N)$. We will use the
bundles $\underline{\Lambda}_{\alpha}$ on $\mathcal{L}^1 N$ and
$\underline{\Lambda}_{(\mathcal{L}\pi)^*\alpha}$ on $\mathcal{L}^1
T^*N$ (see \ref{SubsectionNovikovbundles}), where $\mathcal{L}\pi
\co \mathcal{L}^1 T^*N \to \mathcal{L}^1 N$ is induced by $\pi \co
T^*N \to N$. The twisted complexes
$(CM_*(\mathcal{E};\underline{\Lambda}_{\alpha}),\delta^\mathcal{E})$
and
$(CF^*(H;\underline{\Lambda}_{(\mathcal{L}\pi)^*\alpha}),\delta^H)$
are freely generated over $\Lambda$ respectively by the $q\in
\textnormal{Crit}(\mathcal{E})$ and the $y\in
\textnormal{Crit}(A_H)$, and the twisted differentials are defined
by
$$
\delta^{\mathcal{E}} (q) = \sum_{v \in
\mathcal{M}_0^{\mathcal{E}}(q,p)} \epsilon(v)\, t^{-\alpha[v]} \,
p \quad \quad \quad \quad \delta^{H} (y) = \sum_{u \in
\mathcal{M}_0^{H}(x,y)} \epsilon(u)\,
t^{\alpha[\mathcal{L}\pi(u)]} \, x
$$
since $\alpha[\mathcal{L}\pi (u)] = (\mathcal{L}\pi)^*\alpha[u]$.
The sign difference in the powers of $t$ arises because
$\delta^{\mathcal{E}}$ is a differential and $\delta^{H}$ is a
codifferential. For simplicity, we write $\pi u = \mathcal{L}\pi
(u)$.

\begin{theorem}\label{TheoremAbbondandoloSchwarzTwisted}
If $(L,G,H,J)$ is regular then for all $\alpha \in
H^1(\mathcal{L}N)$ there is a chain-complex isomorphism $\varphi
\co
(CM_*(\mathcal{E};\underline{\Lambda}_{\alpha}),\delta^\mathcal{E})
\to
(CF^{n-*}(H;\underline{\Lambda}_{(\mathcal{L}\pi)^*\alpha}),\delta^H),
$ which on a generator $q$ is defined as
$$
\varphi(q) = \sum_{u_+ \in \mathcal{M}_0^{+}(q,x)} \epsilon(u_+)
\, t^{-\alpha[w] + \alpha[\pi u_+]} \, x,
$$
where $w \co  (-\infty,0] \to \mathcal{L}^1 N$ is the negative
gradient trajectory $w(s)=\phi_{-\nabla
\mathcal{E}}^s(\overline{q})$ connecting $q$ to
$\overline{q}(\cdot)=\pi u_+(0,\cdot)$. The isomorphism is
compatible with the splitting into subcomplexes corresponding to
different conjugacy classes of $\pi_1(N)$, and it is compatible
with the action filtrations: for any $a\in\R$ it induces an
isomorphism on the subcomplexes generated by the $q,x$ with
$\mathcal{E}(q)<a$ and $-A_H(x)<a$.

After identifying Morse cohomology with singular cohomology, the
map $\varphi$ induces an isomorphism
$$
SH^*(DT^*N;\underline{\Lambda}_{\alpha}) \cong H_{n-*}(\mathcal{L}
N;\underline{\Lambda}_{\alpha}).
$$
\end{theorem}
\begin{proof}
Figure \ref{FigureAbbondandoloSchwarz} shows the possible
connected components of $\mathcal{M}_1^{+}(q,x)$. Evaluating
$d\alpha=0$ on a component equals the sum of the values of
$\alpha$ on the broken trajectories. For instance, in the second
figure
$$
-\alpha[w'] + \alpha[\pi u_+'] + \alpha[\pi u'] = -\alpha[v''] -
\alpha[w''] + \alpha[\pi u_+''],
$$
and therefore, since
$\epsilon(u')\epsilon(u_+')=\epsilon(v'')\epsilon(u''_+)$,
$$
\epsilon(u')\epsilon(u_+') \; t^{-\alpha[w'] + \alpha[\pi u_+']}
\; t^{\alpha[\pi u']} = \epsilon(v'')\epsilon(u''_+) \;
t^{-\alpha[v'']} \; t^{-\alpha[w''] + \alpha[\pi u''_+]}.
$$
Thus the broken trajectories contribute equally to
$\delta^{H}(\varphi(q))$ and $\varphi(\delta^{\mathcal{E}}(q))$. A
similar computation shows that in the first and third figures the
broken trajectories contribute zero respectively to
$\varphi(\delta^{\mathcal{E}}(q))$ and $\delta^{H}(\varphi(q))$.
Hence
$$
\begin{array}{l}
\delta^{H}(\varphi(q)) = \displaystyle \sum_{(u',u_+') \in
\mathcal{M}_0^{H}(x,x') \times \mathcal{M}_0^{+}(q,x')}
\epsilon(u')\,
t^{\alpha[\pi u']} \cdot \epsilon(u_+')\, t^{-\alpha[w'] + \; \alpha[\pi u_+']} \, x  \; = \\
\quad = \displaystyle \sum_{(v'',u''_+)\in
\mathcal{M}_0^{\mathcal{E}}(q,p) \times \mathcal{M}_0^{+}(p,x)}
\epsilon(v'')\, t^{-\alpha[v'']} \cdot \epsilon(u''_+)\,
t^{-\alpha[w''] + \alpha[\pi u_+'']} \, x =
\varphi(\delta^{\mathcal{E}}(q)),
\end{array}
$$
where we sum respectively over $x,x'$ and over $x,p$, and where
$w'$, $w''$ are the $-\nabla \mathcal{E}$ trajectories ending in
$\pi u'_+(0,\cdot)$, $\pi u_+''(0,\cdot)$. Hence $\varphi$ is a
chain map.

That $\varphi$ is an isomorphism follows just as in the untwisted
case, because for $\mathcal{E}(q)\leq -A_H(x)$ the only nonempty
$\mathcal{M}_0^{+}(q,x)$ occurs when
$(t,x)=\mathfrak{L}(t,q,\dot{q})$, and in this case
$\mathcal{M}_0^{+}(q,x)=\{u_+\}$ where $u_+\equiv x$ and $w \equiv
q$ are independent of $s \in \R$ and so the coefficient of $x$ in
$\varphi(q)$ is
\[
\epsilon(u_+)\,t^{-\alpha [w] + \alpha [ \pi u_+ ]} =
\epsilon(u_+) = \pm 1.
\]

The last statement in the claim is a consequence of the
identification of the Morse cohomology of $(\mathcal{L}^1
N,\mathcal{E},G)$ with the singular cohomology of $\mathcal{L}^1
N$ just as in \cite{Abbondandolo-Schwarz}, after introducing the
system $\underline{\Lambda}_{\alpha}$ of local coefficients.
\end{proof}


\section{Viterbo Functoriality} \label{SectionViterboFunctoriality}
Let $(M^{2n},\theta)$ be a Liouville domain
(\ref{SubsectionLiouvilleDomainSetup}), and suppose
$$
i \co (W^{2n},\theta') \hookrightarrow (M^{2n},\theta)
$$
is a Liouville embedded subdomain, that is we require that
$i^*\theta - e^\rho\theta'$ is exact for some $\rho \in \R$. For
example the embedding $DT^*L \hookrightarrow DT^*N$, obtained by
extending an exact Lagrangian embedding $L \hookrightarrow DT^*N$
to a neighbourhood of $L$, is of this type. We fix $\delta>0$ with
$$
0<\delta< \textnormal{min} \; \{ \textnormal{periods of the
nonconstant Reeb orbits on } \partial M \textnormal{ and }
\partial W \}.
$$

We will now recall the construction of Viterbo's commutative
diagram (\cite{Viterbo1}):
$$
\xymatrix{ SH^*(W) \ar@{<-}^{c_*}@<1ex>[d]\ar@{<-}[r]^{SH^*(i)} &
SH^*(M)\ar@{<-}^{c_*}@<1ex>[d] \\
H^*(W) \ar@{<-}[r]^{i^*} & H^*(M) }
$$
%
%
%
\subsection{Hamiltonians with small
slopes}\label{SubsectionHamiltoniansWithSmallSlopes}
We now consider Hamiltonians $H^0$ as in
(\ref{SubsectionHamiltoniansLinearAtInfty}), which are $C^2$-close
to a constant on $\hat{M} \setminus (0,\infty) \times
\partial M$; $H^0=h(\x)$ with slopes $h'(\x)\leq \delta$ for
$\x\geq 0$; and which have constant slope $h'(\x)=m>0$ for $\x
\geq \x_0$.

A standard result in Floer homology is that (after a generic
$C^2$-small time-independent perturbation of $(H^0,J)$) the
$1$-periodic orbits of $X_{H^0}$ and the Floer trajectories
connecting them inside $\hat{M} \setminus \{ \x \geq \x_0 \}$ are
both independent of $t\in S^1$, and so these orbits correspond to
critical points of $H^0$ and these Floer trajectories correspond
to negative gradient trajectories of $H^0$. By the maximum
principle, the Floer trajectories connecting these orbits do not
enter the region $\{ \x \geq \x_0 \}$, and by the choice of
$\delta$ there are no $1$-periodic orbits in $\{\x \geq \x_0 \}$
since there $0<h'(\x)\leq \delta$.

The Floer complex $CF^*(H^0)$ is therefore canonically identified
with the Morse complex $CM^*(H^0)$, which is generated by
$\textnormal{Crit}(H^0)$ and whose differential counts the
$-\nabla H^0$ trajectories. The Morse cohomology $HM^*(H^0)$ is
isomorphic to the singular cohomology of $\hat{M}$ (which is
homotopy equivalent to $M$), so
$$
HF^*(H^0) \cong HM^*(H^0) \cong H^*(M).
$$
Moreover, by Morse cohomology, a different choice ${H^0}'$ of
$H^0$ yields an isomorphism $HM^*({H^0}') \cong H^*(M)$ which
commutes with $HM^*(H^0) \cong H^*(M)$ via the continuation
isomorphism $HM^*(H^0) \to HM^*({H^0}')$.
%
%
\subsection{Construction of
$\mathbf{c_*}$}\label{SubsectionConstructionOfc}
%
Recall from (\ref{SubsectionSymplCohAsDirectLimit}) that
$$
SH^*(M) = \lim_{\longrightarrow} HF^*(H),
$$
where the direct limit is over the continuation maps for
Hamiltonians $H$ which equal $h^m_{c,C}(\x)= m(\x -c) + C$ for $\x
\gg 0$, ordered by increasing slopes $m>0$.

Since $H^0$ is such a Hamiltonian, there is a natural map
$HF^*(H^0) \to \lim HF^*(H)$ arising as a direct limit of
continuation maps. By \ref{SubsectionHamiltoniansWithSmallSlopes},
this defines a map
$$
c_* \co  H^*(M) \to SH^*(M).
$$
A different choice ${H^0}'$ yields a map $HF^*({H^0}') \to
SH^*(M)$ which commutes with the map $HF^*(H^0) \to SH^*(M)$ via
the continuation isomorphism $HF^*(H^0) \to HF^*({H^0}')$.
Together with \ref{SubsectionHamiltoniansWithSmallSlopes}, this
shows that $c_*$ is independent of the choice of $H^0$.
%
\subsection{Diagonal-step shaped Hamiltonians}
We now consider the Liouville subdomain $i \co W\hookrightarrow
M$. The $\partial_{r}-$Liouville flow for $\theta'$ defines a
tubular neighbourhood $(0,1+\epsilon) \times \partial W$ of
$\partial W$ inside $\hat{M}$, where $\partial W$ corresponds to
$\x=e^r=1$. This coordinate $\x$ may not extend to
$\hat{M}\setminus W$, and it should not be confused with the $\x$
we previously used to parametrize $(0,\infty)\times \partial M
\subset \hat{M}$.

\begin{figure}
\includegraphics[width=0.6\textwidth]{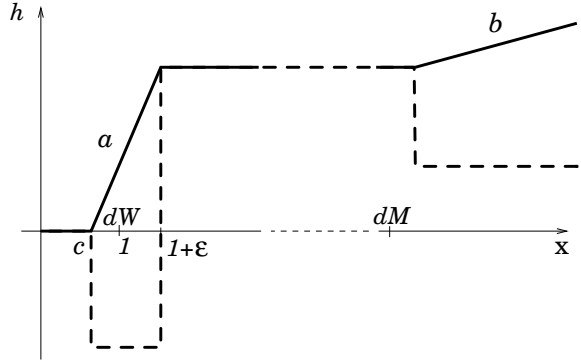}
\caption{The solid line is a diagonal-step shaped Hamiltonian
$h=h_c^{a,b}$ with slopes $a\gg b$. The dashed line is the action
function $A_h(\x)=-\x h'(\x) + h(\x)$.}\label{FigureViterbo}
\end{figure}

We consider diagonal-step shaped Hamiltonians $H$ as in Figure
\ref{FigureViterbo}, which are zero on $W \setminus \{\x \geq c\}$
and which equal $h_c^{a,b}(\x)$ on $\{\x \geq c\}$, where
$h_c^{a,b}$ is piecewise linear with slope $b$ at infinity; with
slope $a \gg b$ on $(c,1+\epsilon)$; and which is constant
elsewhere. We assume that $0 \leq c \leq 1$ and that $a,b$ are
chosen generically so that they are not periods of Reeb orbits
(see \ref{SubsectionReebDynamics}).

As usual, before we take Floer complexes we replace $H$ by a
generic $C^2$-small time-dependent perturbation of it, and the
orbits and action values that we will mention take this into
account. Let $M'\subset \hat{M}$ be the compact subset where $h$
does not have slope $b$. Observe Figure \ref{FigureViterbo}: the
$1$-periodic orbits of $X_H$ that can arise are:
\begin{enumerate}
\item\label{ItemTypeOfOrbits1} critical points of $H$ inside $W
\setminus \{\x \geq c\}$ of action very close to $0$;

\item\label{ItemTypeOfOrbits2} nonconstant orbits near $\x = c$ of
action in $(-ac,-\delta c)$;

\item\label{ItemTypeOfOrbits3} nonconstant orbits near $\x =
1+\epsilon$ of action in $(-ac,a(1+\epsilon-c))$;

\item\label{ItemTypeOfOrbits4} critical points of $H$ in $M'
\setminus (W\cup \{\x\leq 1+\epsilon \})$ of action close to
$a(1+\epsilon-c)$;

\item\label{ItemTypeOfOrbits5} nonconstant orbits near $\partial
M'$ of action $\gg 0$ provided $a \gg b$.
\end{enumerate}

Since the complement of the Reeb periods is open, there are no
Reeb periods in $(a-\nu_a,a+\nu_a)$ for some small $\nu_a>0$. Thus
the actions in case (\ref{ItemTypeOfOrbits3}) will be at least
$$
-(a-\nu_a)(1+\epsilon) + a(1+\epsilon - c) = \nu_a(1+\epsilon) -
ac,
$$
and for sufficiently small $c$, depending on $a$, we can ensure
that this is at least $\nu_a$. Hence (after a suitable
perturbation of $H$) we can ensure that if $a \gg b$ and $c \ll
a^{-1}$ then the actions of (\ref{ItemTypeOfOrbits1}),
(\ref{ItemTypeOfOrbits2}) are negative and those of
(\ref{ItemTypeOfOrbits3}), (\ref{ItemTypeOfOrbits4}),
(\ref{ItemTypeOfOrbits5}) are positive.
%
\subsection{Construction of $\mathbf{SH^*(i)}$}\label{SubsectionConstructionOfSH}

Suppose $H$ is a (perturbed) diagonal-step shaped Hamiltonian,
with $a \gg b$ and $c \ll a^{-1}$ so that the orbits in $W$ have
negative actions and those outside $W$ have positive actions. We
write $CF^*(M,H)$ to emphasize that the Floer complex is computed
for $M$. Consider the action-restriction map
(\ref{SubsectionFloerChainCx})
$$
CF^*(M,H;-\infty,0) \leftarrow CF^*(M,H).
$$
Given two diagonal-step shaped Hamiltonians $H,H'$ with $H\leq H'$
everywhere, pick a homotopy $H_s$ from $H'$ to $H$ which is
monotone ($\partial_s H_s \leq 0$). The induced continuation map
$\phi \co  CF^*(M,H) \to CF^*(M,H')$ restricts to a map on the
quotient complexes $\phi \co  CF^*(M,H;-\infty,0) \to
CF^*(M,H';-\infty,0)$ because the action decreases along Floer
trajectories when $H_s$ is monotone (see
\ref{SubsectionContinuationMaps}).

Consider the Hamiltonian $H_W$ on the completion $\hat{W}=W
\cup_{\partial W} [0,\infty)\times\partial W$ which equals $H$
inside $W$ and which is linear with slope $a$ outside $W$. Then
the quotient complex $CF^*(M,H;-\infty,0)$ can be identified with
$CF^*(W,H_W)$ by showing that there are no Floer trajectories
connecting $1$-periodic orbits of $X_{H}$ in $\hat{M}$ which exit
$W \cup \{ \x \leq 1+\epsilon \}$. Therefore we obtain the
commutative diagram
$$
\xymatrix{ CF^*(W,H_W') \ar@{<-}[r] \ar@{<-}[d] & CF^*(M,H') \ar@{<-}[d] \\
CF^*(W,H_W) \ar@{<-}[r] & CF^*(M,H) }
$$
where the vertical maps are continuation maps and where the
horizontal maps arose from action-restriction maps. Taking
cohomology, and then taking the direct limit as $a \gg b \to
\infty$ (so $c \ll a^{-1} \to 0$) defines the map $SH^*(i)$,
$$
SH^*(i) \co SH^*(W) \leftarrow SH^*(M).
$$
%
%
\subsection{Viterbo
functoriality}\label{SubsectionViterboFunctoriality}
Consider a (perturbed) diagonal-step shaped Hamiltonian $H=H^0$
with $c=1$ and slopes $0< b \ll a < \delta$ so that the orbits
inside $W$ have negative actions and those outside $W$ have
positive actions. Then $H^0$ and the corresponding $H^0_W$ are of
the type described in
(\ref{SubsectionHamiltoniansWithSmallSlopes}) for $M$ and $W$
respectively. The action-restriction map $CF^*(W,H^0_W) \leftarrow
CF^*(M,H^0)$ is then identified with the map on Morse complexes
$CM^*(W,H^0|_W) \leftarrow CM^*(M,H^0)$ which restricts to the
generators $x \in \textnormal{Crit}(H^0)$ with $H^0(x)<0$. In
cohomology this map corresponds to the pullback on singular
cohomology $i^* \co H^*(W) \leftarrow H^*(M)$.

This identifies $CM^*(W,H^0|_W) \leftarrow CM^*(M,H^0)$ with the
bottom map of the diagram in (\ref{SubsectionConstructionOfSH})
when we take $H=H^0$, and so taking the direct limit over the $H'$
we obtain Viterbo's commutative diagram in cohomology:
$$
\xymatrix{ SH^*(W) \ar@{<-}^{c_*}@<1ex>[d]\ar@{<-}[r]^{SH^*(i)} &
SH^*(M)\ar@{<-}^{c_*}@<1ex>[d] \\
H^*(W) \ar@{<-}[r]^{i^*} & H^*(M) }
$$
%
%
%
\subsection{Twisted Viterbo
functoriality}\label{SubsectionTwistedViterboFunctoriality}
We now introduce the twisted coefficients
$\underline{\Lambda}_{\alpha}$ for some $\alpha \in
H^1(\mathcal{L} \hat{M}) \cong H^1(\mathcal{L} M)$, as explained
in (\ref{SubsectionNovikovFloerCohomology}) and
(\ref{SubsectionNovikovSymplecticCohomology}). Recall that we have
constructed twisted continuation maps
(\ref{SubsectionTwistedContinuationMaps}) which compose well, so
the discussion of (\ref{SubsectionConstructionOfc}) and
(\ref{SubsectionConstructionOfSH}) will hold in the twisted case
provided that we understand how the local systems restrict.

Suppose $H^0$ is a Hamiltonian with small slope as in
(\ref{SubsectionHamiltoniansWithSmallSlopes}). In the twisted case
the canonical identification of $CF^*(H^0)$ with the Morse complex
$CM^*(H^0)$ becomes
$$
CF^*(H^0;\underline{\Lambda}_{\alpha}) =
CM^*(H^0;c^*\underline{\Lambda}_{\alpha}),
$$
where $c^*\underline{\Lambda}_{\alpha}$ is the restriction of
$\underline{\Lambda}_{\alpha}$ to the local system on $\hat{M}
\subset \mathcal{L}_0\hat{M}$ which consists of a copy $\Lambda_m$
of $\Lambda$ over each $m \in \hat{M}$ and of the multiplication
isomorphism $t^{\alpha[c \circ v]} = t^{c^*\alpha[v]} \co
\Lambda_m \to \Lambda_{m'}$ for every path $v(s)$ in $\hat{M}$
joining $m$ to $m'$, and where the twisted Morse differential is
defined on $q_{+} \in \textnormal{Crit}(H^0)$ analogously to the
Floer case:
$$
\delta q_{+} = \sum \{ \epsilon(v) \, t^{c^*\alpha[v]} \, q_{-} :
q_{-}\in \textnormal{Crit}(H^0),
\partial_s v = -\nabla H^0(v), \lim_{s \to \pm \infty} v(s) =
q_{\pm} \}.
$$
By mimicking the proof that $HM^*(H^0) \cong H^*(M)$, for twisted
coefficients we have $HM^*(H^0;c^*\underline{\Lambda}_{\alpha})
\cong H^*(M;c^*\underline{\Lambda}_{\alpha})$ (singular cohomology
with coefficients in the local system
$c^*\underline{\Lambda}_{\alpha}$, as defined in
\cite{Whitehead}).

As in (\ref{SubsectionConstructionOfc}), we get twisted
continuation maps $CF^*(H^0;\underline{\Lambda}_{\alpha}) \to
CF^*(H;\underline{\Lambda}_{\alpha})$ for Hamiltonians $H$ linear
at infinity. In cohomology these maps yield a morphism
$HF^*(H^0;\underline{\Lambda}_{\alpha}) \to \lim
HF^*(H;\underline{\Lambda}_{\alpha})$, where the direct limit is
taken over twisted continuation maps as the slopes at infinity of
the $H$ increase. This defines
$$
c_* \co H^*(M;c^*\underline{\Lambda}_{\alpha}) \to
SH^*(M;\underline{\Lambda}_{\alpha}).
$$
In (\ref{SubsectionConstructionOfSH}) we get action-restriction
maps $CF^*(M,H;\underline{\Lambda}_{\alpha};-\infty,0) \leftarrow
CF^*(M,H;\underline{\Lambda}_{\alpha})$, and two choices of
diagonal-step shaped Hamiltonians $H,H'$ with $H\leq H'$ induce a
continuation map $\phi \co  CF^*(M,H;\underline{\Lambda}_{\alpha})
\to CF^*(M,H';\underline{\Lambda}_{\alpha})$ which restricts to
the quotient complexes $\phi \co
CF^*(M,H;\underline{\Lambda}_{\alpha};-\infty,0) \to
CF^*(M,H';\underline{\Lambda}_{\alpha};-\infty,0)$.

Let $\mathcal{L}i \co  \mathcal{L}W \to \mathcal{L}M$ be the map
induced by $i$. As in (\ref{SubsectionConstructionOfSH}), the
quotient complex
$CF^*(M,H;\underline{\Lambda}_{\alpha};-\infty,0)$ can be
identified with
$CF^*(W,H_W;\underline{\Lambda}_{(\mathcal{L}i)^*\alpha})$ because
there are no Floer trajectories connecting $1$-periodic orbits of
$X_{H}$ which exit $W \cup \{ \x \leq 1+\epsilon \}$ in $\hat{M}$
and so the twisted differentials of the two complexes agree since
$(\mathcal{L}i)^* \alpha$ and $\alpha$ agree on the common Floer
trajectories inside $W \cup \{ \x \leq 1+\epsilon \}$.

As in (\ref{SubsectionConstructionOfSH}), the direct limit over
the twisted continuation maps for diagonal-step shaped $H$ of the
action-restriction maps
$$
CF^*(W,H_W;\underline{\Lambda}_{(\mathcal{L}i)^*\alpha})
\leftarrow CF^*(M,H;\underline{\Lambda}_{\alpha})
$$
as $a \gg b \to \infty$ will define a twisted map $SH^*(i)$ in
cohomology,
$$
SH^*(i) \co SH^*(W;\underline{\Lambda}_{(\mathcal{L}i)^*\alpha})
\leftarrow SH^*(M;\underline{\Lambda}_{\alpha}).
$$

As in (\ref{SubsectionViterboFunctoriality}), the
action-restriction maps fit into a commutative diagram
$$
\xymatrix{
CF^*(W,H_W';\underline{\Lambda}_{(\mathcal{L}i)^*\alpha})
\ar@{<-}[r] \ar@{<-}[d] & CF^*(M,H';\underline{\Lambda}_{\alpha}) \ar@{<-}[d] \\
CM^*(W,H^0_W;c^*\underline{\Lambda}_{(\mathcal{L}i)^*\alpha})
\ar@{<-}[r] & CM^*(M,H^0;c^*\underline{\Lambda}_{\alpha}) }
$$
and taking the direct limit over the $H'$ yields the following
result in cohomology.
\begin{theorem}\label{TheoremViterboTwistedFunctoriality}
Let $(M^{2n},\theta)$ be a Liouville domain. Then for all $\alpha
\in H^1({\mathcal{L}M})$ there exists a map $c_* \co
H^*(M;c^*\underline{\Lambda}_{\alpha}) \to
SH^*(M;\underline{\Lambda}_{\alpha})$, where $c \co  M \to
\mathcal{L}M$ is the inclusion of constant loops. Moreover, for
any Liouville embedding $i \co (W^{2n},\theta') \to
(M^{2n},\theta)$ there exists a map $SH^*(i) \co
SH^*(W;\underline{\Lambda}_{(\mathcal{L}i)^*\alpha}) \leftarrow
SH^*(M;\underline{\Lambda}_{\alpha})$ which fits into the
commutative diagram
$$
\xymatrix{ SH^*(W;\underline{\Lambda}_{(\mathcal{L}i)^*\alpha})
\ar@{<-}[r]^-{SH^*(i)} \ar@{<-}[d]^{c_*} &
SH^*(M;\underline{\Lambda}_{\alpha}) \ar@{<-}[d]^{c_*} \\
H^*(W;c^*\underline{\Lambda}_{(\mathcal{L}i)^*\alpha})
\ar@{<-}[r]^-{i^*} & H^*(M;c^*\underline{\Lambda}_{\alpha}) }
$$
\end{theorem}
%

\section{Proof of the Main Theorem}\label{SectionProofMainTheorem}

\begin{lemma}\label{LemmaStandardInclusion}
Let $N^n$ be a closed manifold and let $L \to DT^*N$ be an exact
Lagrangian embedding. Then for all $\alpha \in H^1(\mathcal{L}N)$,
the composite
$$
H_*(N;c^*\underline{\Lambda}_{\alpha})
\stackrel{\sim}{\longrightarrow}
H^{n-*}(N;c^*\underline{\Lambda}_{\alpha})
\stackrel{c_*}{\longrightarrow}
SH^{n-*}(DT^*N;\underline{\Lambda}_{(\mathcal{L}\pi)^*\alpha})
\stackrel{\varphi^{-1}}{\longrightarrow}
H_*(\mathcal{L}N;\underline{\Lambda}_{\alpha})
$$
of Poincar\'e duality, the map $c_*$ from
(\ref{SubsectionTwistedViterboFunctoriality}) and the inverse of
$\varphi$ (Theorem \ref{TheoremAbbondandoloSchwarzTwisted}), is
equal to the ordinary map $c_* \co
H_*(N;c^*\underline{\Lambda}_{\alpha}) \to
H_*(\mathcal{L}N;\underline{\Lambda}_{\alpha})$ induced by the
inclusion of constants $c \co N \to \mathcal{L}N$.
\end{lemma}
In the untwisted case, the lemma was proved by Viterbo
\cite{Viterbo3} using his construction of the isomorphism
$\varphi$, and it can be proved in the Abbondandolo-Schwarz setup
by using small perturbations of $L(q,v)=\frac{1}{2}|v|^2$ and
$H(q,p) = \frac{1}{2}|p|^2$ and by considering the restriction of
the isomorphism $\varphi$ to the orbits of action close to zero.
The twisted version is proved analogously.

\begin{theorem}\label{TheoremMainTheorem}
Let $N^n$ be a closed manifold and let $L \to DT^*N$ be an exact
Lagrangian embedding. Then for all $\alpha \in H^1(\mathcal{L}N)$
there exists a commutative diagram
$$
\xymatrix{
H_{*}(\mathcal{L}L;\underline{\Lambda}_{(\mathcal{L}p)^*\alpha})
\ar@{<-}^{c_*}@<1ex>[d] \ar@{<-}[r]^-{\mathcal{L}p_!} &
H_*(\mathcal{L}N;\underline{\Lambda}_{\alpha}) \ar@{<-}^{c_*}@<1ex>[d] \\
H_*(L;c^*\underline{\Lambda}_{(\mathcal{L}p)^*\alpha})
\ar@{<-}[r]^-{p_!} & H_*(N;c^*\underline{\Lambda}_{\alpha}) }
$$
where $c \co N \to \mathcal{L}N$ is the inclusion of constant
loops, $p \co  L \to T^*N \to N$ is the projection and $p_!$ is
the ordinary transfer map. Moreover, the diagram can be restricted
to the components $\mathcal{L}_0 L$ and $\mathcal{L}_0 N$ of
contractible loops.

If $c^*\alpha=0$ then the bottom map becomes $p_! \otimes 1 \co
H_*(L) \otimes \Lambda \leftarrow H_*(N) \otimes \Lambda$.
\end{theorem}

\begin{proof}
Let $\theta_N$ be the canonical 1-form which makes $(DT^*N , d
\theta_N)$ symplectic. By Weinstein's theorem a neighbourhood of
$L$ is symplectomorphic to a small disc cotangent bundle $DT^*L$.
Therefore the exact Lagrangian embedding $j \co L^n
\hookrightarrow DT^*N$ yields a Liouville embedding $i \co
(DT^*L,\theta_L) \hookrightarrow (DT^*N, \theta_N)$.

By Theorem \ref{TheoremAbbondandoloSchwarzTwisted} there are
twisted isomorphisms
$$
\begin{array}{ll}
\varphi_N \co  H_*(\mathcal{L}N;\underline{\Lambda}_{\alpha}) \to
SH^{n-*}(DT^*N;\underline{\Lambda}_{(\mathcal{L}\pi)^*\alpha})\\
\varphi_L \co
H_*(\mathcal{L}L;\underline{\Lambda}_{(\mathcal{L}p)^*\alpha}) \to
SH^{n-*}(DT^*L;\underline{\Lambda}_{(\mathcal{L}i)^*(\mathcal{L}\pi)^*\alpha})
\end{array}
$$
We define $\mathcal{L}p_{!} = \varphi_L^{-1} \circ SH^*(i) \circ
\varphi_N$ so that the following diagram commutes
$$
\xymatrix{
H_{*}(\mathcal{L}L;\underline{\Lambda}_{(\mathcal{L}p)^*\alpha})
\ar@{<-}_{\varphi_L^{-1}}^{\wr}@<1ex>[d]
\ar@{<-}[r]^-{\mathcal{L}p_!} &
H_*(\mathcal{L}N;\underline{\Lambda}_{\alpha}) \ar@{->}_{\varphi_N}^{\wr}@<1ex>[d] \\
SH^{n-*}(DT^*L;\underline{\Lambda}_{(\mathcal{L}i)^*(\mathcal{L}\pi)^*\alpha})
\ar@{<-}[r]^-{SH^*(i)} &
SH^{n-*}(DT^*N;\underline{\Lambda}_{(\mathcal{L}\pi)^*\alpha}) }
$$
Recall that the ordinary transfer map $p_!$ is defined using
Poincar\'e duality and the pullback $p^*$ so that the following
diagram commutes,
$$
\xymatrix{
H^{n-*}(L;c^*\underline{\Lambda}_{(\mathcal{L}p)^*\alpha})
\ar@{<-}^{\wr}@<1ex>[d] \ar@{<-}[r]^-{p^*} &
H^{n-*}(N;c^*\underline{\Lambda}_{\alpha}) \ar@{->}^{\wr}@<1ex>[d] \\
H_*(L;c^*\underline{\Lambda}_{(\mathcal{L}p)^*\alpha})
\ar@{<-}[r]^-{p_!} & H_*(N;c^*\underline{\Lambda}_{\alpha}) }
$$
Finally, Theorem \ref{TheoremViterboTwistedFunctoriality} for the
map $i$ yields another commutative diagram whose horizontal maps
are the bottom and top rows respectively of the above two diagrams
(in the second diagram we use that $L,N$ are homotopy equivalent
to $DT^*L$, $DT^*N$). By combining these diagrams we obtain a
commutative diagram
$$
\xymatrix{
H_{*}(\mathcal{L}L;\underline{\Lambda}_{(\mathcal{L}p)^*\alpha})
\ar@{<-}@<1ex>[d] \ar@{<-}[r]^-{\mathcal{L}p_!} &
H_*(\mathcal{L}N;\underline{\Lambda}_{\alpha}) \ar@{<-}@<1ex>[d] \\
H_*(L;c^*\underline{\Lambda}_{(\mathcal{L}p)^*\alpha})
\ar@{<-}[r]^-{p_!} & H_*(N;c^*\underline{\Lambda}_{\alpha}) }
$$
Lemma \ref{LemmaStandardInclusion} shows that the vertical maps
are indeed the maps $c_*$ in ordinary homology. Since $c \co  N
\to \mathcal{L}N$ maps into the component of contractible loops
$\mathcal{L}_0 N$, the diagram restricts to $\mathcal{L}_0 L$ and
$\mathcal{L}_0 N$ by restricting $\mathcal{L}p_{!}$ and projecting
to
$H_{*}(\mathcal{L}L;\underline{\Lambda}_{(\mathcal{L}p)^*\alpha})$
(not all loops in $T^*L$ that are contractible in $T^*N$ need be
contractible in $T^*L$).
\end{proof}


\section{Proof of the Corollary}\label{SectionProofMainCorollary}
%
\subsection{Transgressions}\label{SubsectionTransgressions}
Given $\beta \in H^2(N)$, let $f \co N \to \CP$ be a classifying
map for $\beta$. Let $ev \co  \LN \times S^1 \to N$ be the
evaluation map. Define
$$\xymatrix{
\tau=\pi\circ ev^* \co  H^2(N)\ar[r]^-{ev^*} & H^2(\LN \times
S^1)\ar[r]^-{\pi} & H^1(\LN) },
$$
where $\pi$ is the projection to the K\"{u}nneth summand. If $N$
is simply connected, then $\tau$ is an isomorphism. Let $u$ be a
generator of $H^2(\CP)$, then $v=\tau(u)$ generates
$H^1(\mathcal{L}\CP) \cong H^1(\Omega\CP)$ and
$\tau(\beta)=(\mathcal{L}\!f)^*v$. Identify $H^1(\LN)\cong
\textnormal{Hom}(\pi_1(\LN),\Z)$ and $\pi_1(\LN) \cong
\pi_2(N)\rtimes \pi_1(N)$, then the class $\tau(\beta)$ vanishes
on $\pi_1(N)$ and corresponds to
$$
f_* \co  \pi_2(N) \to \pi_2(\CP)\cong \Z.
$$
Similarly, define $\tau_b \co  H^2(N)\to H^1(\Omega_0 N)$ for the
space $\Omega_0 N$ of contractible based loops. Then $\Omega f \co
\Omega_0 N \to \Omega \CP$ is a classifying map for
$\tau_b(\beta)$. The inclusion $\Omega_0 N \to \LN$ induces a
bijection $\tau(\beta) \mapsto \tau_b(\beta)$ between transgressed
forms.

We will assume throughout that the transgression
$\alpha=\tau(\beta) \in H^1(\LN)$ is nonzero, or equivalently that
$f_* \co  \pi_2 (N) \to \Z$ is not the zero map.
%
\subsection{Novikov homology of the free loopspace}\label{SubsectionNovikovHomology}
Denote by $\overline{\LN}$ the infinite cyclic cover of $\LN$
corresponding to $\alpha \co  \pi_1 (\LN) \to \Z$, and let $t$
denote a generator of the group of deck transformations of
$\overline{\LN}$. The group ring of the cover is $R=\Z[t,t^{-1}]$,
and $\Lambda = \Z((t))=\Z[[t]][t^{-1}]$ is the Novikov ring of
$\alpha$ (see \ref{SubsectionNovikovbundles}).

The Novikov homology of $\LN$ with respect to $\alpha$ is defined
as the homology of $\LN$ with local coefficients in the bundle
$\underline{\Lambda}_{\alpha}$, which by \cite{Whitehead} can be
calculated as
$$
H_*(\LN;\underline{\Lambda}_{\alpha}) \cong
H_*(C_*(\overline{\LN})\otimes_R \Lambda).
$$

Say that a space $X$ is of \emph{finite type} if $H_k(X)$ is
finitely generated for each $k$.
\begin{theorem}\label{TheoremCyclicCoverFG}
For a compact manifold $N$, if $\tau(\beta)\neq 0$ and $\pi_m(N)$
is finitely generated for each $m\geq 2$ then $\overline{\LN}$ is
of finite type.
\end{theorem}

\begin{proof}
\textbf{Claim 1.} \emph{If $\overline{\Omega_0 N}$ is of finite
type then so is
$\overline{\LN}$.}\\
Proof. Consider the fibration $\Omega_0 N \to \LN \to N$, and take
cyclic covers corresponding to $\tau_b(\beta)$ and $\tau(\beta)$
to obtain the fibration $\overline{\Omega_0 N} \to \overline{\LN}
\to N$. By compactness, $N$ is homotopy equivalent to a finite CW
complex and Claim 1 follows by a Leray-Serre spectral sequence
argument.\\

After replacing $N$ by a homotopy equivalent space, we may assume
that we have a fibration $f \co N \to \CP$ with fibre
$F=f^{-1}(*)$, and taking the spaces of contractible based loops
gives a fibration $\Omega f \co
\Omega_0 N \to \Omega \CP$.\\

\textbf{Claim 2.} \emph{The fibre of $\Omega f$ is a union
$(\Omega F)_K$ of finitely many components of $\Omega F$, indexed
by the finite set
$K=\textnormal{Coker}(f_* \co \pi_2 N \to \pi_2 \CP)$.}\\
Proof. Consider the homotopy LES for the fibration $f$,
$$
\xymatrix{
\pi_2 N \ar@{->}[r]^-{f_*}& \pi_2 \CP \ar@{->}[r]& \pi_1 F
\ar@{->}[r]& \pi_1 N
}
$$
then $(\Omega f)^{-1}(*)= \Omega F \cap \Omega_0 N$ consists of
loops $\gamma\in\Omega F$ whose path component lies in the kernel
of $\pi_1 F\to\pi_1 N$, which is isomorphic to the cokernel of
$f_*$. Since $\tau(\beta)\neq 0$, also $f_*$ is nonzero and so $K$
is finite.
\\

\textbf{Claim 3.} \emph{$\overline{\Omega j} \co (\Omega F)_K \to
\overline{\Omega_0 N}$
is a homotopy equivalence.}\\
Proof. Observe that $\overline{\Omega_0 N}$ is the pull-back under
$\Omega f$ of the cyclic cover of $\Omega\CP$ corresponding to the
transgression $v=\tau_b(u)\in H^1(\Omega\CP)$ of a generator $u\in
H^2(\CP)$ (see \ref{SubsectionTransgressions}). We obtain the
commutative diagram
$$
\xymatrix{
(\Omega F)_K \ar@{->}[d]\ar@{->}[r]^-{\overline{\Omega j}} &
\overline{\Omega_0 N} \ar@{->}[d]\ar@{->}[r]^{\overline{\Omega f}}
&
\overline{\Omega \CP}\ar@{->}[d]\ar@{->}[r]^-{\overline{\varphi}}_-{\simeq} & \R\ar@{->}[d]\\
(\Omega F)_K \ar@{->}[r]^-{\Omega j} & \Omega_0 N
\ar@{->}[r]^-{\Omega f} & \Omega\CP
\ar@{->}[r]^-{\varphi}_-{\simeq} & S^1}
$$
where the homotopy equivalence $\varphi$ corresponds to $\tau_b(u)
\in H^1(\Omega\CP)\cong [\Omega\CP,S^1]$. The claim follows since
$\R$ is contractible.\\

\textbf{Claim 4.} \emph{$\overline{\Omega_0 N}$ is of finite type
iff
$\Omega_0 F=\Omega \widetilde{F}$ is of finite type.}\\
Proof. Each component of $\Omega F$ is homotopy equivalent to
$\Omega_0 F$ via composition with an appropriate fixed loop. The
claim follows from Claims 3 and 2 since $K$ is finite. Note that
we may identify $\Omega_0 F=\Omega \widetilde{F}$ since the loops
of $F$ that lift to closed loops of the universal cover
$\widetilde{F}$ are precisely the contractible ones.\\

\textbf{Claim 5.} \emph{$\Omega \widetilde{F}$ is of finite type
iff $\pi_m N$
is finitely generated for each $m\geq 2$.}\\
Proof. Since $\widetilde{F}$ is simply connected,
$\Omega\widetilde{F}$ is of finite type iff $\widetilde{F}$ is of
finite type, by a Leray-Serre spectral sequence argument applied
to the path-space fibration $\Omega\widetilde{F}\to
P\widetilde{F}\to\widetilde{F}$ (see \cite[9.6.13]{Spanier}).
Moreover $\widetilde{F}$ is of finite type iff
$\pi_m(\widetilde{F})=\pi_m(F)$ is finitely generated for all
$m\geq 2$ (see \cite[9.6.16]{Spanier}). The claim follows from the
homotopy LES for $F \to N \to \CP$.
\end{proof}

\begin{corollary}\label{CorollaryVanishingTheorem}
For a compact manifold $N$, if $\tau(\beta)\neq 0$ and $\pi_m(N)$
is finitely generated for each $m\geq 2$, then
$H_*(\LN;\underline{\Lambda}_{\tau(\beta)})=0$.
\end{corollary}
\begin{proof}
We need to show that each $HN_k =
H_k(\LN;\underline{\Lambda}_{\tau(\beta)})$ vanishes. Since
$\Z[t]$ is Noetherian, its $(t)-$adic completion $\Z[[t]]$ is flat
over $\Z[t]$ (see \cite[Theorem 8.8] {Matsumura}). Therefore,
localizing at the multiplicative set $S$ generated by $t$,
$\Lambda = S^{-1}\Z[[t]]$ is flat over $R=S^{-1}\Z[t]$. Thus $HN_k
\cong H_k(\overline{\LN})\otimes_R \Lambda$, which is the
localization of $H_k=H_k(\overline{\LN})\otimes_{\Z[t]} \Z[[t]]$.
Observe that $t\cdot H_k = H_k$ since $t$ acts invertibly on
$H_k(\overline{\LN})$. So if $H_k$ were finitely generated over
$\Z[t]$, then $H_k=0$ by Nakayama's lemma \cite[Theorem 2.2]
{Matsumura} since $t$ lies in the radical of $\Z[[t]]$. By Theorem
\ref{TheoremCyclicCoverFG}, $H_k$ is in fact finitely generated
over $\Z$, so this concludes the proof.
\end{proof}

\begin{remark}
The idea behind the proof of Corollary
\ref{CorollaryVanishingTheorem} is not original. I later realized
that it is a classical result that if $H_*(X;\Z)$ is finitely
generated in each degree then the Novikov homology
$H_*(C_*(\overline{X})\otimes_R \Lambda_{\alpha})$ vanishes for
$0\neq \alpha \in H^1(X)$. The basic idea dates back to
\cite{Milnor} and a very general version of this result is proved
in \cite[Prop. 1.35]{Farber}.
\end{remark}

\begin{corollary}\label{Corollary2VanishingTheorem}
If $N$ is a compact simply connected manifold, then
$H_*(\LN;\underline{\Lambda}_{\alpha})=0$ for any nonzero
$\alpha\in H^1(\LN)$.
\end{corollary}
\begin{proof}
$N$ is simply connected so its homotopy groups are finitely
generated in each dimension because its homology groups are
finitely generated by compactness (see \cite[9.6.16]{Spanier}).
Since $N$ is simply connected, any $\alpha$ in $H^1(\LN)$ is the
transgression of some $\beta\in H^2(N)$. The result now follows
from Corollary \ref{CorollaryVanishingTheorem}.
\end{proof}
%
\subsection{Proof of the Main Corollary}\label{SubsectionMainCorollary}
\begin{corollary}\label{CorollaryMainCorollaryProof}
Let $N^n$ be a closed simply connected manifold. Let $L \to DT^*N$
be an exact Lagrangian embedding. Then the image of $p_* \co
\pi_2(L) \to \pi_2(N)$ has finite index and $p^* \co H^2(N)\to
H^2(L)$ is injective.
\end{corollary}
\begin{proof}
A non-zero class $\beta\in H^2(N)$ yields a non-zero transgression
$\tau(\beta)\in H^1(\LN)$ (see \ref{SubsectionTransgressions}).
Suppose by contradiction that $\tau(p^*\beta)=0$. Then the local
system $(\mathcal{L}p)^*\underline{\Lambda}_{\tau(\beta)}$ is
trivial (see \ref{SubsectionNovikovbundles}). Moreover
$c^*\tau(\beta)=0$ since $\tau(\beta)$ vanishes on $\pi_1(N)$.
Therefore the diagram of Theorem \ref{TheoremMainTheorem},
restricted to contractible loops, becomes
$$
\xymatrix{ H_{*}(\LL) \otimes \Lambda \ar@{<-_{)}}_{c_*}@<-1ex>[d]
\ar@{<-}[r]^{\mathcal{L}p_!} &
H_*(\LN;\underline{\Lambda}_{\tau(\beta)}) \ar@{<-}^{c_*}@<1ex>[d] \\
H_*(L) \otimes \Lambda \ar@{<<-}_{q_*}@<-1ex>[u]\ar@{<-}[r]^{p_!}
& H_*(N) \otimes \Lambda }
$$
where $q \co \mathcal{L}_0 L \to L$ is the evaluation at $0$. By
Corollary \ref{Corollary2VanishingTheorem},
$H_*(\LN;\underline{\Lambda}_{\tau(\beta)})=0$, so the fundamental
class $[N] \in H_n(N)$ maps to $c_*[N]=0$. But
$\mathcal{L}p_!(c_*[N])=c_*p_{!}[N]=c_*[L]\neq 0$ since $c_*$ is
injective on $H_*(L)$.

Therefore $\tau(p^*\beta)$ cannot vanish, and so $\tau_b \circ p^*
\co  H^2 N \to H^1(\Omega L)$ is injective. Consider the
commutative diagram
$$
\xymatrix{ H^2(N) \ar@{->}^{p^*}[d]\ar@{->}[r]^-{\tau_b}_-{\sim} &
\textnormal{Hom}(\pi_2(N),\Z) \cong H^1(\Omega N) \ar@{->}@<-5ex>^{(\Omega p)^*}[d] \\
H^2(L) \ar@{->}[r]^-{\tau_b} & \textnormal{Hom}(\pi_2(L),\Z) \cong
H^1(\Omega L) }
$$
where the top map $\tau_b$ is an isomorphism since $N$ is simply
connected. We deduce from the injectivity of $\tau_b \circ
p^*=(\Omega p)^*\circ \tau_b$ that $p^* \co H^2(N)\to H^2(L)$ and
$\textnormal{Hom}(\pi_2(N),\Z) \to \textnormal{Hom}(\pi_2(L),\Z)$
are both injective, so in particular the image of $p_* \co
\pi_2(L) \to \pi_2(N)$ has finite index.
\end{proof}

%
\section{Non-simply connected cotangent bundles}
\label{SubsectionNonSimplyConnectedCotangentBundles}
%

We will prove that for non-simply connected $N$ the map $\pi_2(L)
\to \pi_2(N)$ still has finite index provided that the homotopy
groups $\pi_m(N)$ are finitely generated for each $m \geq 2$.

This time we consider transgressions induced from the universal
cover $\widetilde{N}$ of $N$,
$$\tau \co  H^2(\widetilde{N}) \to H^1(\mathcal{L}\widetilde{N})=H^1(\LN)\cong \textrm{Hom}(\pi_2 N,\Z).
$$

The homomorphism $ \widetilde{f}_* \co  \pi_2(\widetilde{N})=
\pi_2(N) \to \Z$ corresponding to such a transgression
$\tau(\widetilde{\beta})$ is induced by a classifying map
$\widetilde{f} \co  \widetilde{N}\to \CP$ for
$\widetilde{\beta}\in H^2(\widetilde{N})$. Since $\Omega
\widetilde{N} = \Omega_0 N$ and $\mathcal{L} \widetilde{N} = \LN$,
the transgressions $\tau_b(\widetilde{\beta})$ and
$\tau(\widetilde{\beta})$ define cyclic covers $\overline{\Omega_0
N}$ and $\overline{\LN}$. We will use these in the construction of
the Novikov homology.

\begin{theorem}
Let $N$ be a compact manifold with finitely generated $\pi_m(N)$
for each $m\geq 2$. If $\tau(\widetilde{\beta})\neq 0$ then
$\overline{\LN}$ is of finite type and
$H_*(\LN;\underline{\Lambda}_{\tau(\widetilde{\beta})})=0$.
\end{theorem}

\begin{proof}
Revisit the proof of Theorem \ref{TheoremCyclicCoverFG}. It
suffices to prove that $\overline{\Omega_0 N}$ has finite type.
This time we have the commutative diagram
$$
\xymatrix{
\Omega F \ar@{->}[d]\ar@{->}[r]^-{\overline{\Omega j}} &
\overline{\Omega_0 N} \ar@{->}[d]\ar@{->}[r]^{\overline{\Omega
\widetilde{f}}} &
\overline{\Omega \CP}\simeq \R \ar@{->}[d]\\
\Omega F \ar@{->}[d]\ar@{->}[r]^-{\Omega j} & \Omega
\widetilde{N}=\Omega_0 N
\ar@{->}[d]\ar@{->}[r]^-{\Omega \widetilde{f}} & \Omega\CP \simeq S^1\ar@{->}[d]\\
F \ar@{->}[r]^-{j} & \widetilde{N} \ar@{->}[r]^-{\widetilde{f}} &
\CP}
$$
Since $\Omega F \simeq \overline{\Omega_0 N}$, it suffices to show
that $\Omega F$ has finite type. Observe that
$$\Omega F \cong \oplus_K \Omega_0 F$$
where $K = \textnormal{Coker}(\widetilde{f}_* \co \pi_2 N \to
\pi_2 \CP)$ is a finite set since $\widetilde{f}_*\neq 0$. So we
just need to show that $\Omega_0 F = \Omega \widetilde{F}$ is of
finite type. The same argument as in Theorem
\ref{TheoremCyclicCoverFG} proves that $\Omega \widetilde{F}$ is
of finite type iff $\pi_m N=\pi_m \widetilde{N}$ is finitely
generated for each $m\geq 2$. The same proof as for Corollary
\ref{CorollaryVanishingTheorem} yields the vanishing of the
Novikov homology.
\end{proof}

\begin{corollary}
Let $N$ be a closed manifold with finitely generated $\pi_m(N)$
for each $m\geq 2$. Let $L \to DT^*N$ be an exact Lagrangian
embedding. Then the image of $p_* \co \pi_2(L) \to \pi_2(N)$ has
finite index and $\widetilde{p}^* \co H^2(\widetilde{N}) \to
H^2(\widetilde{L})$ is injective.
\end{corollary}

\begin{proof}
The proof is analogous to that of Corollary
\ref{CorollaryMainCorollaryProof}: $(\mathcal{L} p)^*$ in the
diagram
$$
\xymatrix{ H^2(\widetilde{N})
\ar@{->}^{\widetilde{p}^*}[d]\ar@{->}[r]^-{\tau}_-{\sim} &
\textnormal{Hom}(\pi_2(N),\Z) \cong H^1(\LN) \ar@{->}@<-5ex>^{(\mathcal{L} p)^*}[d] \\
H^2(\widetilde{L}) \ar@{->}[r]^-{\tau}_-{\sim} &
\textnormal{Hom}(\pi_2(L),\Z) \cong H^1(\LL) }
$$
is injective because if, by contradiction,
$\tau(\widetilde{p}^*\widetilde{\beta}) \in H^1(\LL)$ vanished
then the functoriality diagram of Theorem \ref{TheoremMainTheorem}
would not commute.
\end{proof}

%
\section{Unoriented theory}
\label{SubsectionUnorientedTheory}
%
So far we assumed that all manifolds were oriented. By using
$\Z_2=\Z/2\Z$ coefficients instead of $\Z$ coefficients one no
longer needs the Floer and Morse moduli spaces to be oriented in
order to define the differentials and continuation maps. For the
twisted setup, we change the Novikov ring to
$$
\Lambda = \Z_2((t))=\Z_2[[t]][t^{-1}],
$$
the ring of formal Laurent series with $\Z_2$ coefficients. The
bundle $\underline{\Lambda}_{\alpha}$ is now a bundle of
$\Z_2((t))$ rings, however the singular cocycle $\alpha$ is still
integral: $[\alpha] \in H^1(\LN;\Z)$.

Using these coefficients, all our theorems hold true without the
orientability assumption on $N$ and $L$. The following is an
interesting application of Corollary
\ref{CorollaryMainCorollaryProof} in this setup.

\begin{corollary}
There are no unorientable exact Lagrangians in $T^*S^2$.
\end{corollary}
\begin{proof}
For unorientable $L$, $H^2(L;\Z)=\Z_2$. Therefore the
transgression $\tau$ vanishes on $H^2(L;\Z)$ since its range
$\textrm{Hom}(\pi_2(L),\Z)$ is torsion-free. But for $S^2$ there
is a non-zero transgression. This contradicts the proof of
Corollary \ref{CorollaryMainCorollaryProof}.
\end{proof}

%
%
%
\bibliographystyle{gtart}

\end{document}